\documentclass[11pt]{article} 
\usepackage{amsmath}
\usepackage{lscape}
\usepackage[section]{placeins}
\usepackage{mathtools}
\usepackage{amsthm}
\usepackage{times}
\usepackage{graphicx}
\usepackage{amssymb}
\usepackage{url} 
\usepackage{array}
\usepackage{microtype}
\usepackage{authblk}
\usepackage{dsfont}
\usepackage{siunitx}
\usepackage{minitoc}
\usepackage[toc,page,header]{appendix}
\usepackage{subcaption}
\usepackage[margin=1in]{geometry}
\usepackage{longtable}
\usepackage{booktabs}
\usepackage{multirow}
\usepackage{xcolor}
\usepackage{comment}
\usepackage{color,soul} 
\usepackage{graphicx} 
\usepackage{caption}
\usepackage{algorithm}
\usepackage{algorithmic}

\usepackage{tabularx}
\usepackage{array}
\usepackage{booktabs} 
\usepackage{booktabs} 
\definecolor{darkgreen}{RGB}{0,90,0.1}
\usepackage[colorlinks=true, citecolor=darkgreen, urlcolor=darkgreen]{hyperref}
\usepackage{authblk}

\usepackage{tikz}
\usetikzlibrary{calc}
\usetikzlibrary{arrows.meta,positioning,fit,backgrounds}
\usepackage{multicol}
\usepackage{dcolumn}
\usepackage{enumitem}
\usepackage[capitalize,noabbrev]{cleveref}
\usepackage[version=4]{mhchem} 
\usepackage[textsize=tiny]{todonotes}
\usepackage{natbib}
\theoremstyle{plain}
\newtheorem{theorem}{Theorem}[section]
\newtheorem{proposition}[theorem]{Proposition}
\newtheorem{lemma}[theorem]{Lemma}
\newtheorem{corollary}[theorem]{Corollary}
\theoremstyle{definition}

\theoremstyle{remark}
\newtheorem{remark}[theorem]{Remark}

\title{Novel technique based on Léja Points Approximation for Log-determinant Estimation of Large   matrices}
\author[1,3]{Verlon Roel Mbingui \thanks{Email: vmbingui@aimsric.org/ vrmbingui@aimsammi.org}}
\author[2,4]{Antoine Tambue\thanks{Email: antonio@aims.ac.za}}
\author[3]{Issa Karambal\thanks{Email: issa@aims.ac.za}}
\affil[1]{\textit{Department of Mathematics and Statistics, The University of Dodoma, Dodoma, P.O. Box 338, Tanzania}}
\affil[2]{\textit{Department of Computing Mathematics and Physics, Western Norway University of Applied
Sciences, Inndalsveien 28, 5063, Bergen, Norway} } 
\affil[3]{\textit{African Institute for Mathematical Sciences, Research and Innovation Centre (AIMSRIC), Kigali, Rwanda }}
\affil[4]{\textit{Department of Mathematics and Applied Mathematics, University of Cape Town, 7701, Rondebosch,South Africa }} 
\date{\today}

\begin{document}

\maketitle

\newcommand{\fix}{\marginpar{FIX}}
\newcommand{\new}{\marginpar{NEW}}

\begin{abstract}

The computation of the Log-determinant of large, sparse, symmetric positive definite (SPD) matrices is essential in many scientific computational fields such as numerical linear algebra and machine learning. In low dimensions, Cholesky is preferred, but in high dimensions, its computation may be prohibitive due to memory limitation. To circumvent this, Krylov subspace techniques have proven to be efficient but may be computationally expensive due to the required orthogonalization processes. In this paper, we introduce a Newton-Léja interpolation based approach  for estimating log-determinants, where the implementation is only based on matrix multiplications and a rough estimation of eigenvalue bounds of the  matrix. The method employs  an adaptive stopping criterion to select the degree of interpolation according to the data and introduces a centered, scaled-based formulation for computing the logarithmic divided differences.  By combining Newton-Léja  interpolation with a randomized algorithm called Hutch++, we obtain a matrix-free log-determinant estimation framework that avoids explicit matrix factorization and Krylov basis orthogonalization. We  establish approximation error bounds for the Newton-Léja interpolation  and derive an inexact oracle error bound that separates stochastic trace  estimation from the approximation errors in the projected and residual components. 

Numerical experiments on both large sparse synthetic matrices (Gaussian Markov Random fields) and large-scale real-world matrices demonstrate that our method substantially reduces computational time while retaining competitive accuracy, thereby confirming its effectiveness, scalability, and favorable accuracy-cost trade-off for large-scale log-determinant estimation compared with stochastic Lanczos quadrature (SLQ) and related stochastic trace-estimation methods. 
\end{abstract}

\textbf{keywords:} Log Determinant, Léja points, Divided Differences, Hutch++, Large matrices

\section{Introduction}
\label{sec1}
\quad
The computation of the Log-determinant of the large, symmetric definite matrix (SPD) has numerous applications in numerical linear algebra  with very wide-ranging applications in data science and scientific computation. In machine learning, the Log-determinant appears in the evaluation of the log-marginal likelihood of the  Gaussian processes \cite{RasmussenWilliams2005,dong2017logdet,zhang2007approximate}, in  normalization constant of the Gaussain Markov Random Fields (GMRFs) \cite{RueHeld2005,wainwright2006log} and in the training of deep kernels models \cite{wilson2016deepkernel,dong2017logdet}.  In scientific computing, they appear in uncertainty quantification, Bayesian inverse problems, statistical physics, etc. \cite{aune2014parameter,saibaba2017trace,chung2024efficient,latz2025deep}. However, the  challenge with the exact computation of the Log-determinant using, for instance, Cholesky decomposition requires $\mathrm{O}(n^3)$ operations and  $\mathrm{O}(n^2)$ storage  which becomes prohibitive as the dimension $n$  of the matrix increases \cite{xiang2010error,aune2014parameter,han2015large,zhang2007approximate}. 

Since the computation of the Log-determinant is equivalent to the computation of the trace of the log of the matrix, research has been focused on finding efficient approach to compute the trace instead. One of these approaches  is  stochastic  estimation of the trace. The Hutchinson method \cite{hutchinson1989stochastic} is the classical approach for estimating the Log-determinant. It uses random Rademacher vectors or Gaussian probe vectors to get an unbiased estimate. However, to achieve the $(1\pm \epsilon) $ approximation to the trace, $\mathrm{O}(1/\epsilon^2)$ matrix-vector multiplications are required, which can be computationally expensive. Tang et al. \cite{tang2012probing} developed a structured probing method based on graph coloring, which exploits sparsity patterns to estimate the diagonals of matrix inverses more efficiently and, by extension, provides improved trace,  
 but suffers from significant limitations: perform poorly on unstructured problems, requires many costly linear solutions, and is sensitive to conditioning. Recent advances, such as Hutch++ \cite{meyer2021hutch++}, have significantly improved the classical approach. By combining a low-rank randomized approximation with variance reduction, Hutch++ reduces the number of matrix-vector multiplications required to $\mathrm{O}(1/\epsilon )$ to achieve a relative error of $\epsilon$. 
The stochastic  Lanczos quadrature (SLQ) \cite{ubaru2017fast} takes advantage of Krylov subspaces to approximate spectral sums, often outperforming Hutchinson’s estimator when the matrix is well-conditioned. 

However, Aune et al. \cite{aune2014parameter} evaluated the action of the matrix logarithm on a vector using a contour-integral representation of the logarithm. This approach involves solving a sequence of linear systems with shifted versions of the matrix, each corresponding to a point on the integration contour. Although this is theoretically sound, it is computationally demanding because the matrix shifts vary with each contour point, preventing efficient reuse of the factorizations. Moreover, the choice of integration contour requires prior knowledge of the spectrum, the numerical quadrature adds extra complexity, and the resulting linear systems may become unstable if the contour approaches eigenvalues. Alternatively, polynomial and rational approximation techniques have been proposed. Chebyshev polynomial expansions \cite{han2015large, han2017approximating}, for example, approximate the logarithm function over a spectral interval with near-minimax polynomials, offering predictable convergence, but requiring accurate spectral bounds and careful scaling. Despite Chebyshev's approach is well known for its stability and efficacy in polynomial interpolation, SLQ has outperformed it in various types of practical application\cite{ubaru2017fast}.

In this work, we propose a novel technique based on real Léja points interpolation. Real Léja points interpolation has gained attention as a stable alternative for approximating matrix functions. Real Léja sequences maximize the spread of interpolation points, mitigating oscillations, and enabling efficient recursive Newton interpolation. Unlike Chebyshev points, which are not nested for different polynomial degrees, Léja points form a nested sequence that can be progressively refined by adding points without having to recalculate from scratch, allowing for adaptive refinement to arbitrary granularity \cite{narayan2014adaptive,deka2022exponential,deka2025lexint}. This means that interpolation can be increased until the desired accuracy is achieved, which is ideal for large-scale adaptive algorithms.

In fact, as the number of observations increases, Léja points on a real interval cluster asymptotically like Chebyshev extreme points, allowing for an equivalent quasi-minimal approximation without the Runge phenomenon \cite{breuss2016numerical}.  Even with very high polynomial degrees, Breuss et al. in \cite{breuss2016numerical} confirm that using Léja points in particular,  Fast Léja achieves the same high degree of interpolation accuracy as Chebyshev points while being simpler to use. Applications to exponential matrix functions \cite{tambue2010exponential, tambue2013efficient,caliari2004interpolating} and related $\varphi$-functions have shown remarkable stability and accuracy. However, despite this potential, to the best of our knowledge, the application of Léja points interpolation to the matrix logarithm and, by extension, the log determinant estimation has not been systematically explored. 
The goal of this paper is to develop and analyze a Léja points interpolation-based method for Log-determinant estimation. Our approach is to replace the usual polynomial or rational approximation for $\log(Q)v$ with a Newton-Léja interpolation scheme, tailored specifically for the logarithm. The contribution of this article can be summarized as follows :

\begin{enumerate}

   \item[(i)] We develop a Newton-Léja polynomial approximation framework for computing the action of the matrix logarithm on a vector. The method employs    real Léja points generated in the fixed interval $[-2,2]$ and affinely mapped to the spectral interval $[\lambda_{min},\lambda_{max}]$ of a SPD matrix, enabling the   approximation of $\log(Q)v$ without requiring eigen decomposition or matrix factorization.

  \item [(ii)]We introduce a  centered, scaling-based approach for computing the logarithmic divided differences required by by Newton-Léja interpolation. This formulation provides an alternative procedure for evaluating the interpolation coefficients through a truncated Taylor expansion. 
  
 \item [(iii)]  We integrate the newton Léja approximation into the Hutch++ stochastic trace estimator to approximate the log-determinant of large sparse symmetric positive definite matrices. The resulting method combines polynomial matrix-function approximations with randomized trace estimation without requiring explicit construction of the matrix logarithm. 
\item[(iv)] We establish error bounds for the Newton-Léja approximation and the resulting log-determinant estimator. In particular, we derive an error decomposition for inexact Léja-Hutch++, separating the stochastic trace estimation error from the approximation errors associated with the projected and residual matrix-logarithm actions.  
\item[(v)] Our experiments show that  the combination of Léja points interpolation  with Hutch++  yields significant speedups  over SLQ \cite{ubaru2017fast},  So-SLQ \cite{han2023suboptimal} (combinaison of Lanczos quadrature and Hutch++), funNyström++ \cite{cortinovis2022randomized} and \\ Krylov-aware \cite{chen2023krylov} with competitive accuracy on large sparse problems.  
\end{enumerate}
The paper is organized as follows: We present the related work on Log-determinant in \Cref{log det approx}. \Cref{approximation} describes the Léja points interpolation  approach used for  computing the action of the logarithm of the matrix to a vector. A novel approach to compute the divided differences is also provided along with  error bounds for the Léja points approximation.  \Cref{Léja_hutch_sec} describes the algorithm to estimate  the trace via Léja points interpolation and Hutch++ and estimates the  error bounds  of  the approximation obtained by this method. In \Cref{experiments}, we evaluate the performance of our proposed  Léja-based Log-determinant estimator through extensive numerical experiments, demonstrating its competitiveness.

\section{Related works}
\label{log det approx}
Let $Q \in \mathbb{R}^{n\times n}$ denote a symmetric positive matrix (SPD). Then the logarithm of the determinant of $Q$ is defined as
\begin{equation}
    \log(\det Q) = \mathrm{tr}(\log (Q )),
    \label{Léja1}
\end{equation}
where 
\begin{equation}
 \mathrm{tr}(\log Q) =\sum_{j=1}^{n}e_j^T\log(Q)e_j,
    \label{Léja2}   
\end{equation}
with $e_j$ a standard basis vector. 
For a large matrix $Q$, computing the logarithm of the determinant from the right-hand side of  \eqref{Léja2} becomes computationally expensive as $n$ grows. Instead, we consider approximating it using the Monte Carlo method, i.e., 
\begin{equation}
    \mathrm{tr}(\log(Q))\approx \dfrac{1}{m_{vec}}\sum_{j=1}^{m_{vec}}v_j^T\log(Q)v_j,
\label{Léja8}
\end{equation}
where $v_j\in\mathbb{R}^n$ is a vector whose entries are drawn from $\{-1,1\}$ with equal probability, and $m_{vec}<n$.
 
One of the existing approaches typically rely on Hutchinson’s stochastic trace estimator\cite{bai1996some}, which yields a Monte Carlo approximation of the Log-determinant. Given its Monte Carlo nature, confidence intervals can be constructed using Hoeffding's inequality \cite{bai1996bounds,baglama1998fast,aune2014parameter}. This inequality can also provide guidance for selecting an appropriate sample size $m_{vec}$. Although  this approach is memory-efficient, it necessitates a sufficiently large sample size to ensure accurate estimation. 
The computational cost can be significant, potentially leading to lengthy computation times for a single - determinant approximation \cite{aune2014parameter}.
Using probing vectors \cite{tang2012probing}, it is possible to achieve accurate results with fewer vectors than those required by Monte Carlo methods \cite{aune2014parameter}.  Aune et al., Bekas et al. and Tang et al \cite{aune2014parameter,bekas2007estimator,tang2012probing}  have pioneered the use of probing vectors to extract the diagonal elements of matrices and their inverses. Bekas et al. \cite{bekas2007estimator} have specifically focused on extracting the diagonal elements of sparse matrices under certain constraints. Tang et al. \cite{tang2012probing} employ an approximate sparsity pattern of the inverse of $Q$, determined by the power of the matrix $Q$ ($Q^p, p = 1, 2, 3, \dots $). The Hermite polynomial interpolation \cite{higham2008functions} confirms the accuracy of this approach for sufficiently large powers of $p$. However, such approaches may become computationally impractical in large-scale settings. In many applications, the  inverse of  the SPD matrix $Q$ and well known matrix function such as $ \log(Q)$  exhibits a significant decay, a lower power of $p$ may suffice. This observation naturally leads to a graph-based interpretation: each matrix $Q$ can be associated with a graph whose structure reflects its sparsity pattern,  often be well-approximated by a matrix sharing the same sparsity pattern as $ Q^p $, enabling the design of efficient probing vectors. However, finding an appropriate value of $p$ is nontrivial. For selecting the power $p$ in $Q^p$,
 Tang and Saad \cite{tang2012probing}  suggested a heuristic way based on the decay of entries in the solution of a linear system which can be computationally expensive for large-scale problems. 
Specifically, one solves
\begin{equation}
  Qx = e_j,
  \label{eq_ prob-graph_col}
\end{equation}
and determines the smallest integer $p$ satisfying
$$
p = \min \left\{ d(\ell,j) \;\middle|\; |x_\ell| < \varepsilon \right\},
$$
where $d(\ell,j)$ denotes the graph distance between points $\ell$ and $j$, $\varepsilon$ is a prescribed tolerance, and $e_j$ denotes the $j$-th canonical basis vector. 
This criterion exploits the relationship between graph distance and the decay of off-diagonal entries in $Q^{-1}$. 

In the other hand, to estimate the action of $\log(Q)$ on a vector $v_j$, Aune et al. \cite{aune2014parameter} numerically approximated the Cauchy’s integral 
\begin{equation*}
\frac{1}{2\pi i}\oint_{\Gamma} \log(z)\,(zI-Q)^{-1}v_j\,dz,
\label{eq_cauchy-log}
\end{equation*}
for a certain contour $\Gamma$ enclosing the spectrum of $Q$, using a quadrature formula. 
However, due to the geometry of the contour, their approach proved inefficient. To address this issue, Hale et al. \cite{hale2008computing} applied a conformal contour transformation, followed by a trapezoidal rule method, to achieve rapid convergence rates through reliable estimation of the smallest and largest eigenvalues of the corresponding matrix. 
This approach naturally yields a rational approximation to the logarithm, where $\log(Q)v_j$ is approximated by a finite sum of shifted resolvent evaluations. 
\[
(\sigma_\ell I-Q)x = v_j,
\]
with complex shifts $\sigma_\ell$. 
In the context of large-scale systems, the problem is posed through the use of Krylov subspace methods. The work of Aune et al. makes use of the shift-invariance property held by the Krylov subspace to derive all the shifted solutions simultaneously without the use of any matrix-vector products.  Despite the theoretical attractiveness of this approach, this methodology also has some issues regarding its practicality :  the process of solving several shifted linear systems comes with a substantial amount of computational complexity and sensitivity to how well the available preconditioners work; having complex shifts will make the process of implementing and analyzing the numerical stability of the scheme more complicated, mostly due to the proximity of the eigenvalues of the matrix $Q$ to the integration contour.

Motivated by these limitations, in this paper, we propose a novel method that is matrix-free and relies on two main ingredients : 
\begin{enumerate}
    \item computing the action of a matrix logarithm $\log(Q)$ on a vector $v$ using Newton interpolation at real Léja points.
    \item The Hutch++ algorithm \cite{meyer2021hutch++} for variance-reduced stochastic trace estimation.
\end{enumerate}

\section{Approximating the  matrix-vector product $\log(Q)v$  to $P_m(Q)v$ using Real fast Léja points technique}
\label{approximation}
To approximate the action of the matrix logarithm on a vector, we employ Newton interpolation at Léja points. 
For a given probing vector $v$, Real fast Léja points approximate $\log(Q)v$ by $P_m(Q)v$ instead of computing the full matrix function $\log (Q)$,  where $P_m$ is an interpolation polynomial of degree $m$ of $\log$ at the sequence of points $\{\xi_i\}^{m}_{i=0}$ called spectral real fast Léja points. These points $\{\xi_i\}^{m}_{i=0}$  belong to the spectral focal interval $[\lambda_{min}, \lambda_{max} ]$ of the matrix $Q$, i.e the focal interval of the smaller ellipse containing all eigenvalues of $Q$. This spectral interval can be estimated by the well-known Gershgorin circle theorem \cite{thomas2013numerical}. 

For a real
interval $[\lambda_{min}, \lambda_{max}]$, a sequence of real fast Léja points $\{\xi_i\}^m_{i=0}$ is defined recursively as follows. Given an initial point $\xi_0$, usually  $\xi_0 = \lambda_{max} $, the sequence of fast Léja points is generated by :
    \begin{equation}
\prod_{k=0}^{j-1}\vert \xi_j - \xi_k \vert = \max_{\xi \in [\lambda_{min}, \lambda_{max}]}\prod_{k=0}^{j-1}\vert \xi - \xi_k \vert \qquad j = 1,2,3, \dots.
        \label{Léja5}
    \end{equation}
We use the Newton's form of the interpolating polynomial $P_m$  given by 
\begin{equation}
P_m(x)= f [ \xi_0 ] + \sum_{j=1}^{m}f [ \xi_0, \xi_1, \cdots, \xi_j]\prod_{k=0}^{j-1}( x - \xi_k)
    \label{Léja6}
\end{equation}
where $f[\bullet]$ are divided differences. 
\begin{remark}  Since Gershgorin circle theorem  depends on diagonal dominance \cite{deville2019optimizing}, not on definiteness, the lower bound of the matrix may be negative. In particular, if some rows have relatively small diagonal entries compared to the sum of the off-diagonal magnitudes. Thus, Gershgorin bounds are robust and not expensive; they are generally too conservative and may misleadingly suggest nonpositive eigenvalues even for SPD matrices. To remedy this, in practice, we replace the negative value with an $\epsilon >0$. An alternative approach  to estimate the spectral, is to compute  $\lambda_{max}$ using the Lanczos iteration \cite{lanczos1950iteration,saad2011numerical} and $\lambda_{min}$ using the shift-invert strategy \cite{ericsson1980spectral}. This method does not suffer from  the limitation of the Gershgorin circle theorem; however, it is more generally computational expensive. 
\end{remark}
In this work, we propose a novel computational technique for divided differences.
\subsection{Proposed alternative method for computing the divided differences of the logarithm function}
\label{subsec_ prop_scheme_dd}
The  divided difference $f[\bullet]$ are defined recursively by
\begin{equation}
\displaystyle
\begin{cases}
f[\xi_j]=f(\xi_j) \\
f[\xi_j,\xi_{j+1}] = \dfrac{f(\xi_{j+1})-f(\xi_{j})}{\xi_{j+1}-\xi_{j}}\\
f[\xi_j, \xi_{j+1}, \cdots, \xi_k] : = \dfrac{f[\xi_{j+1}, \xi_{j+2}, \cdots, \xi_k]-f[\xi_j, \xi_{j+1}, \cdots, \xi_{k-1}]}{\xi_k-\xi_j}.
\end{cases}
\label{Léja7}
\end{equation}
Due to cancellation errors, the standard technique for computing the divided differences may not produce accurate results at machine precision. It can also be numerically unstable, especially when the interpolation points are close. To solve those issues, we propose a scaling technique.
Indeed, given that Léja points in the interval $[-2,2]$ produce optimal accuracy (cf.~\cite{reichel1990newton}), we consider the map 
\begin{equation}
\label{z_mapping}
[-2,2]\ni  \xi\mapsto z=c+\gamma\xi\in[\lambda_{min}, \lambda_{max}]
\end{equation}
Here $\lambda_{min}$ and $\lambda_{max}$ are the minimum and maximum eigenvalues of the matrix $Q$, respectively, and
\begin{equation}
 c=(\lambda_{min} + \lambda_{max})/2 
 \label{eq_centre}
\end{equation} and 
\begin{equation}
 \gamma =(\lambda_{max}- \lambda_{min} )/4.  
 \label{eq_gamma}
\end{equation}
Let $s_{val}>0$ be a scaling factor. To improve the stability, observe that 
\begin{equation}
\label{log_decomposition}
    \log(z)=\log(s_{val})+\log\bigl(1+(z/s_{val}-1)\bigr)
\end{equation}
Thus, the right-hand side of the above equation admits a Taylor expansion if $|\tfrac{z}{s_{val}}-1|\le1$. Consequently, 
\begin{equation}
    s_{val}\ge\frac{\lambda_{max}}{2}.
\end{equation}
The following lemma gives the optimal choice of $s_{val}$.
\begin{lemma}
\label{lem_ Minimax optimality}
Let $s_{val}>0$. Then the optimal scaling value $s_{val}^\star$ satisfying the Taylor expansion condition is given by
\begin{equation}
    \begin{split}
      s^{\star}_{val}&=\min_{s_{val}>0}\big\{\max_{z\in[\lambda_{min},\lambda_{max}]}|z/s_{val}-1| \bigr\} \\
      &=(\lambda_{min} + \lambda_{max})/2 
    \end{split}
\end{equation}
At the optimal choice, one has
$$
|z/s^{\star}_{val}-1|\le \dfrac{\lambda_{max}-\lambda_{min}}{\lambda_{max}+\lambda_{min}}\le1.
$$
\end{lemma}
\begin{proof}
For any fixed $s_{val}>0$, the function $z \mapsto \big \vert \dfrac{z}{s_{val}}-1\big \vert$ is convex in $z$, so the maximum over $[\lambda_{min},\lambda_{max}]$ is reached at one of the endpoints. Set
$$
D(s_{val}) = \max\left\{\bigl|\dfrac{\lambda_{min}}{s_{val}}-1\bigr|,\bigl|\dfrac{\lambda_{max}}{s_{val}}-1\bigr|\right\}.
$$
We now minimize this quantity over $s_{val}>0$. First we consider $s_{val} \in [\lambda_{min}, \lambda_{max}] $. Then 
$$ \vert \lambda_{min}-s_{val}\vert = s_{val}-\lambda_{min} \quad \text{and} \quad \vert \lambda_{max}-s_{val}\vert = \lambda_{max}- s_{val}.$$
The optimal choice is achieved when the two endpoint errors are equal in magnitude,
$$
1-\dfrac{\lambda_{min}}{s_{val}} = \dfrac{\lambda_{max}}{s_{val}}-1,
$$
resulting in $s_{val}^\star=(\lambda_{min}+\lambda_{max})/2$. Substituting back gives
$$
D(s_{val}^\star) = \frac{\lambda_{max}-\lambda_{min}}{\lambda_{max}+\lambda_{min}}.
$$
Finally, if $s_{val} < \lambda_{min}$ then
$$ D(s_{val})=\dfrac{\lambda_{max} -s_{val}}{s_{val}}= \dfrac{\lambda_{max}}{s_{val}}-1. $$
Taking the derivative of the previous expression, we get
$$ D'(s_{val})= - \dfrac{\lambda_{max}}{s^2}<0, \quad \text{since}\quad \lambda_{max}>0.$$
Thus in $(0,\lambda_{min})$,  $D(s_{val})$ strictly decreases. So the largest $s_{val}$ in the region is $s_{val}=\lambda_{min}$.

Similarly, if $s_{val}>\lambda_{max}$, then
$$ D(s_{val})=\dfrac{s_{val}-\lambda_{min} }{s_{val}}= 1-\dfrac{\lambda_{min}}{s_{val}}. $$

Taking the derivative in the previous expression, 
$$ D'(s_{val})=  \dfrac{\lambda_{min}}{s^2}>0, \quad \text{since}\quad \lambda_{min}>0.$$
Then in the region where $s>\lambda_{max}$, $D(s_{val})$ is  strictly increasing,  so the best choice is at $s_{val}=\lambda_{max}$. Therefore the global minimizer lies in $[\lambda_{min}, \lambda_{max}]$ and it is $s_{val}^\star$.
\end{proof}
We now describe our approach, which enables a stable computation of the divided differences associated with the logarithm function. Let $w_i = z_i/s_{val}-1$. Since $|w|\le1$ for $s_{val}\ge\lambda_{max}/2$, we can estimate the evaluation of $\log$ in \eqref{log_decomposition} at an interpolation point $z_i=c+\gamma\xi_i$, with $\xi_i$ being Léja points
\begin{equation*}
    f(z_i)=\log(z_i)\approx \log(s_{val})+\sum_{k=1}^{p} (-1)^{k+1} \frac{w^k_i}{k}.
\end{equation*}
In the matrix form, the divided differences are given by the first column of the matrix function $\log(Q_m)$ (cf.~\cite{geiger2012exponential})
$$
\log(Q_m) e_1 \approx \bigl( \log(s_{val}) I + \sum_{k=1}^{p} (-1)^{k+1} \frac{W_m^k}{k} \bigr) e_1, 
$$

where, $e_1 \in \mathbb{R}^{m+1} ~ ~ (\text{the first standard basis vector }),$ 
$$
W_m = \frac{1}{s_{val}}(Q_m - s_{val} I) ~ \text{and} ~ Q_m = cI-\gamma \widehat{Q}_m,
$$
with 
\begin{equation}
    \widehat{Q}_m =  
\begin{pmatrix}
\xi_0 &   \\
1 & \xi_1 & \\
&1& \ddots &  \\
&&\ddots & \ddots &  \\
&&&\ddots & \ddots &  \\
&&&&1& \xi_n
\end{pmatrix}.
\label{eq_leja_matrix}
\end{equation}

In practice, the real fast Léja points are calculated once within the interval $[- 2, 2]$ (see eg. \cite{geiger2012exponential}) and reused at each time step during the computation of the divided differences. We employ the efficient algorithm from \cite{baglama1998fast} to precompute a large number of real fast Léja points in the interval $[- 2, 2]$. We summarized our approach in the following Algorithm.
\\
\begin{algorithm}[H]
 \caption{Compute of stable divided differences $\mathbf{d}$ of $\log(c+\gamma\xi)$ at Léja points using matrix-function formulation.}
\label{algo:dd_stable}
\begin{algorithmic}
 \STATE \textbf{Input:}{ $Q, z,p$ \{matrix,  number of Léja points to be generated Point, Taylor order\}}
    \STATE 1. $\left[\lambda_{min}, \lambda_{max} \right]= \mathrm{getfocal}(Q)$ (\text{ using Gershgorin Circle theorem \cite{thomas2013numerical} })) 
 \STATE 2. $\xi = \mathrm{getLeja}(-2, 2, z)$  (generate $z$   Léja points  from Eq. \eqref{Léja5})
\STATE 3. $c=(\lambda_{min}+\lambda_{max})/2,\quad \gamma=(\lambda_{max}-\lambda_{min})/4$

\vspace{0.1cm}
\STATE 4. Construct the Léja matrix $\widehat{Q}_m$ using Eq. \eqref{eq_leja_matrix}
\STATE  5. Form the matrix $Q_m = cI + \gamma \widehat{Q}_m  $
\STATE 6. Choose $s_{val}$ 
\STATE Form the matrix $W_m=\dfrac{1}{s_{val}}\left(Q_m-s_{val}I  \right)$
\STATE 7. Compute $\log(I+W_m)$ using Taylor expansion of order $p$
\STATE 8. Compute $\log(Q_m) e_1 = \bigl( \log(s_{val}) I + \log(I + W_m) \bigr) e_1$
\STATE \textbf{Output: } $\mathbf{d}= \log(Q_m) e_1 $.
\end{algorithmic}
\end{algorithm}
The following Algorithm computes the action of the matrix function $\log(Q)$ on a vector  $v$.
In our implementation, we estimate  the focal interval for $Q$ only once and precompute a sufficiently large number $\xi$ of Léja points using the efficient algorithm
 of \cite{baglama1998fast} for a focal interval of $Q$. 
 The divided differences are computed using \Cref{algo:dd_stable}.

\begin{algorithm}
 \caption{Compute $\log(Q)v$ with real fast Léja points. Error $e_m$ is controlled to a prescribed tolerance $tol$ so that $e_m^{Leja} < tol.$}
\label{algoLéja}
\begin{algorithmic}
 \STATE \textbf{Input:}{ $Q, v, tol, \mathbf{d}$ \{matrix, vector, tolerance, divided differences\}}

\vspace{0.1cm}
\STATE 1. Compute the divided differences :
       $$\mathbf{d}=\{d_0,\ldots,d_{z-1}\}~~\text{ of} ~~\log \text{at the points}~~
       \{c+\gamma\xi_i\}_{i=0}^{z-1}~ \text{\Cref{algo:dd_stable}}.$$

\STATE  2. $w_0=v, P_0 = d_0w_0, m=0$ \{initialisation\}
 
   \WHILE{$e_m^{\text{Léja}} = \lvert d_m \rvert \cdot \lVert w_m \rVert_2 > tol$}
    \STATE $w_{m+1} = (Q-\xi_m I)w_m$
    \STATE $m = m+1$
  
  \STATE $P_m(Q)v = P_{m-1} + d_m w_m$\;
\ENDWHILE
\STATE \textbf{Output: } $ P_m(Q)v $.

\end{algorithmic}
\end{algorithm}
\newpage
In the following proposition, we derive error bounds for the Léja points approximation, which gives the convergence rate.
\begin{proposition}
\label{thm : approximation}
  Let $K = [\lambda_{min},\lambda_{max}]$ such that $0< \lambda_{min}<\lambda_{max}$, and $\kappa=\dfrac{\lambda_{max}}{\lambda_{min}}$. Let $P_m$ be the $m$- degree Newton-Léja interpolant of $f(x)=\log(x)$ given by \eqref{Léja6}. Denote by $\Lambda_{m+1}$, the Léja interpolation Lebesgue constant. Then, for every $\rho>1$ with 
    $ \rho=\dfrac{\sqrt{\kappa}+1}{\sqrt{\kappa}-1},$
there exists a constant $C_\rho>0$ such that 
$$\Vert \log -P_m \Vert_\infty \le (1+\Lambda_{m+1})C_\rho \rho^{-m}.$$
It then follows that
\begin{equation}
 \lim \sup _{m\rightarrow \infty}\Vert \log -P_m \Vert_\infty ^{1/m}\le \rho^{-1}.   \label{eq_lebesgue leja1}
\end{equation}
Furthermore, for any $v\in \mathbb{R}^n$, we have
\begin{equation}
\limsup_{m\rightarrow\infty}\Vert \log(Q)v-P_m(Q)v\Vert_2^{1/m}\le \dfrac{1}{\rho}. 
\label{eq_action_error}
\end{equation}
\end{proposition}
\begin{proof}
   By the Lebesgue inequality  for interpolation \cite{taylor2010lebesgue}, we have
\begin{equation}
\|L_{m+1}f-f\|_\infty\le(1+\Lambda_{m+1})E_m(f),    \label{eq_lebesgue leja2} 
\end{equation} 
where $E_m(f)$ denotes the best uniform approximation error of $f$ by polynomials at most $m$ in $K$, and $\Lambda_m$ is the Lebesgue constant associated with the chosen interpolation method \cite{ibrahimoglu2016lebesgue,taylor2010lebesgue}.
Applying this with $f(x)=\log(x)$ and $L_{m+1}f=P_m$, gives
$$ \Vert \log -P_m \Vert_\infty \le (1+\Lambda_{m+1})E_m(\log;[\lambda_{min},\lambda_{max}]).$$
We apply an affine transformation that maps the spectral interval $ [\lambda_{min},\lambda_{max}]$ onto $t\in [-1,1]$ by $x=c+ at$ with 
$$ c= \dfrac{\lambda_{min}+\lambda_{max}}{2},~~ a = \dfrac{\lambda_{max}-\lambda_{min}}{2},$$
and define $g(t)=\log(x+at).$ Since $g(t)$ is analytic in the Bernstein ellipse $\mathcal{E}_\rho$, classical Bernstein Chebyshev approximation theory \cite{berrut2004barycentric,han2015large}, yields 
\begin{equation}
    E_m(\log,[\lambda_{min},\lambda_{max}])\le \dfrac{4M_\rho}{(\rho-1)\rho^{m}}, \quad \text{with}  \quad \rho=\dfrac{\sqrt{\kappa}+1}{\sqrt{\kappa}-1},
    \label{eq_approximation_log_leja1}
\end{equation}
This implies
\begin{equation}
 E_m(\log,[\lambda_{min},\lambda_{max}])\le C_\rho\rho^{-m}, \qquad \text{with} \quad C_\rho = \dfrac{4M_\rho}{(\rho-1)}. \label{eq_lebesgue leja3}  
\end{equation}
By combining \eqref{eq_lebesgue leja2} and \eqref{eq_lebesgue leja3} , we have
\begin{equation}
    \Vert \log -P_m \Vert_\infty \le (1+\Lambda_{m+1})C_\rho \rho^{-m}. \label{eq_lebesgue leja4}
\end{equation}
It is has been proven by Taylor and Totik in \cite[Theorem 1.2 ]{taylor2010lebesgue} that, for the first $m$ points in Léja sequence, $\Lambda_m^{1/m}\rightarrow 1 ~~\text{as}~~ m \rightarrow \infty$ (sub exponential).
Taking the $m$-th root $\lim\sup$ and using  $(1+\Lambda_m)^{1/m}\rightarrow 1$ , \eqref{eq_lebesgue leja4} yields \eqref{eq_lebesgue leja1}. This proves the geometric convergence in the $m$-root sens of Newton Léja of $\log$, with asymptotic exponential  rate  governed by $\rho^{-1}.$   Equation \eqref{eq_action_error} follows from \eqref{eq_lebesgue leja1}.
\end{proof}
\section{Estimating $\mathrm{tr}(\log Q)$ via Real-Léja Interpolation and Hutch ++}
\label{Léja_hutch_sec}
In this section, we describe our proposed algorithm for estimating the Log-determinant of a symmetric positive definite matrix using Hutch ++ which is originally designed for positive semi-definite (PSD) matrices. Suppose the matrix $Q$ is SPD. Then $\log(Q)$ is negative for all $\lambda_{min} \in (0,1)$. This could, therefore, lead to negative eigenvalues. To ensure that this is not the case, that is, that $\log(Q)$ is indeed SPD, we choose $\sigma > 0$ such that:

$$\sigma \le \lambda_{min}.$$
We then define the normalized matrix as follows:
\begin{equation}
\widetilde{Q}\coloneqq
\begin{cases}
    \frac{1}{\sigma}Q,&\textnormal{if}\,\lambda_{min}(Q)<1\\
    Q,&\textnormal{otherwise}
    
\end{cases}.
\label{scaled_sigma_matrix}
\end{equation}
Consequently, when $\lambda_{min}(Q)<1$, we have
$$\lambda_{min}(\widetilde{Q})\ge 1.$$
Thus, $\log(\widetilde{Q})$ is guaranteed to be PSD, and so Hutch++ can be applied. From \eqref{scaled_sigma_matrix}, we have, when $\lambda_{min}(Q)<1$,
\begin{equation}
   \log(\det(Q)) = n\log(\sigma)+\mathrm{tr}( \log(\widetilde{Q})).
\end{equation}
Therefore, to compute $\log(\det(Q))$, it suffices to efficiently estimate $\mathrm{tr}(\log(\widetilde{Q}))$. In the next section, we will describe our approach based on Hutch++ algorithm.
\subsection{Trace estimation with Huth++}
\label{sub_trace_est_hutchpp}
Let $S \in \mathbb{R}^{n\times \frac{1}{3}m_{vec}}$ (with $m_{vec}$, the number of queries)  be a random matrix whose entries $\{+1,-1\}$ are drawn at random with a probability of $\dfrac{1}{2}$. 
Given $S$, we construct the following matrix
\begin{equation}
Y = \log(\widetilde{Q}) S.
\label{eq_lejdetc_onstruction}
\end{equation}
Then, we set 
\begin{equation*}
   A = \operatorname{orth}(Y),
\end{equation*}
where $\operatorname{orth}(Y)$ is the orthonormal matrix corresponding to the QR decomposition of $Y$.
In addition, we define the orthogonal projector 
\begin{equation} 
P = AA^\top
\label{eq-proj_leja}
\end{equation} and 
 decomposes the trace as
\begin{equation}
\mathrm{tr}(\log(\widetilde{Q}))
= \mathrm{tr}(P\log(\widetilde{Q})P) + \mathrm{tr}((I-P)\log(\widetilde{Q})(I-P)).
\label{eq_lej_trace_decomp}
\end{equation}
In \eqref{eq_lej_trace_decomp}, the first term is deterministically evaluated (low-rank part), and the second term represents
the residual stochastic contribution and it is estimated via a randomized  Hutchinson scheme. 
Since $A^\top A = I$, the deterministic low-rank term  can be simplified to
\begin{equation}
\mathrm{tr}(P\log(\widetilde{Q})P)
= \mathrm{tr}(A^\top \log(\widetilde{Q}) A)
= \sum_{i=1}^{\frac{1}{3}m_{vec}} v_i^\top \log(\widetilde{Q}) v_i,
\label{eq_lej_deterministic_part}
\end{equation}
where $v_i $ is the $ith$ column of $A$. Each $\log(\widetilde{Q})v_i$ is approximated using  \Cref{algoLéja}.\\[0.2cm]

The residual term in \Cref{eq_lej_trace_decomp} can be written as
\begin{equation}
\mathrm{tr}_{\mathrm{res}}
= \mathrm{tr}((I-P)\log(\widetilde{Q})(I-P))
\simeq \mathbb{E}[v^\top (I-P)\log(\widetilde{Q})(I-P)v].
\label{eq_res_part_hutch}
\end{equation}
This implies
\begin{equation}
    \mathrm{\widehat{tr}}_{res} = \dfrac{3}{m_{vec}}\sum_{i=j}^{\frac{1}{3}m_{vec}}v_j^\top (I-AA^\top) \log(\widetilde{Q})(I-AA^\top)v_j = \dfrac{3}{m_{vec}}\sum_{i=j}^{\frac{1}{3}m_{vec}}u_j^\top \log(\widetilde{Q})u_j,
    \label{eq_lej_residual_def}
\end{equation}
with $\mathrm{\widehat{tr}}_{res}$ the approximation of \eqref{eq_res_part_hutch}, and 
\begin{equation}
 u_j = (I-AA^\top)v_j,
 \label{eq_res_vectors}
\end{equation}
   the residual correcting additional Rademacher or Gaussian vectors $v_j$. Each $\log(\widetilde{Q})u_j$, is approximated using  \Cref{algoLéja}. 
Combining both  \eqref{eq_lej_deterministic_part} and \eqref{eq_lej_residual_def}, we get 
\begin{equation}
   \widehat{\mathrm{tr}( \log(\widetilde{Q}))} \approx \mathrm{tr}(P\log(\widetilde{Q})P) + \mathrm{\widehat{tr}}_{res}.
    \label{eq_lejdeteq}
\end{equation}
The method can be summarized by the following algorithm :
\newpage
\begin{algorithm}
 \caption{Log determinant estimation of a SPD matrix using Léja points and Hutch++}
\label{algo:log_det_leja}
\begin{algorithmic}
 \STATE \textbf{Input:}{SPD matrix, $Q\in \mathbb{R}^{n\times n}$. The number of queries, $m_{vec}$}. Spectral interval, $[\lambda_{min},\lambda_{max}]$.
 \STATE 1. Define a scaled matrix $\widetilde{Q}\in \mathbb{R}^{n\times n} $ as in \eqref{scaled_sigma_matrix}. 
    \STATE 2. Sample sketching matrix $S\in \mathbb{R}^{n\times \frac{1}{3}m_{vec}}$ a Rademacher (or Gaussian) random matrix with i.i.d $\{+1,-1\} $ with probability $\dfrac{1}{2}$.
\STATE 3. Compute $Y=FS = [\log(\widetilde{Q})S_1,\dots,\log(\widetilde{Q})S_{m_s}]$ using the Algorithm \ref{algoLéja}.
\STATE 4. Form an orthonormal basis $A$ ($A \in \mathbb{R}^{n\times  \frac{1}{3}m_{vec}}$) for the span of $Y$ via $QR$ decomposition.
\STATE 5. Compute the small deterministic term $\mathrm{tr}(A^\top \log(\widetilde{Q}) A)$ by applying $\log(\widetilde{Q})$ to the column of $A$  and approximate each $\log(\widetilde{Q})v_i$ using Algorithm \ref{algoLéja}.
\STATE 6. Estimate the residual correcting additional random vectors $u_j = (I-AA^T)v_j$ :
\begin{equation*}
    \mathrm{\widehat{tr}}_{res} = \dfrac{3}{ m_{vec}}\sum_{j=1}^{ \frac{1}{3}m_{vec}}u_j^\top \log(\widetilde{Q})u_j,
\end{equation*}
where $\log(\widetilde{Q})z_j$  are approximated using Algorithm \ref{algoLéja}.
\STATE 7. $ \widehat{\mathrm{tr}( \log(\widetilde{Q}))}\approx \mathrm{tr}(A^\top F A) + \mathrm{\widehat{tr}}_{res}.$
\STATE \textbf{Output: } $ \widehat{\log(\det(Q))} \approx n\log(\sigma)+\widehat{\mathrm{tr}( \log(\widetilde{Q}))}$
\end{algorithmic}
\end{algorithm}

\begin{theorem}
\label{thm_hpp_psd}
If Hutch++ is implemented with \\ $ m_{vec}=\mathcal{O}\left( \sqrt{\log(1/\delta)}/\epsilon+\log(1/\delta)\right) $ matrix-vector multiplication queries, then for any PSD $B = \log(\widetilde{Q})$ with probability $\geq 1-\delta$, the Hutch++ estimator $\widehat{\mathrm{tr}( \log(\widetilde{Q}))}$ applied to $\log(\widetilde{Q})$ satisfies : 
\begin{equation}
   (1-\epsilon)\mathrel{tr}(\log( \widetilde{Q}))\leq \widehat{\mathrm{tr}( \log(\widetilde{Q}))} \leq (1+\epsilon)\mathrm{tr}(\log( \widetilde{Q})).
   \label{eq_Hutc++ psd}
\end{equation}
\end{theorem}
 \begin{proof}
   Since $\log(\widetilde{Q})$ is PSD, by applying the Hutch++ to $\log(\widetilde{Q})$ (Theorem 1.1 \cite{meyer2021hutch++}), we have \eqref{eq_Hutc++ psd}.  
\end{proof}
\begin{corollary}
Let $\widetilde{Q}$ be defined as in \eqref{scaled_sigma_matrix} with $Q\in\mathbb{R}^{n\times n}$ a symmetric positive definite matrix. Then with probability at least $1-\delta$, we have
\begin{equation}
   (1-\epsilon)\mathrm{tr}(\log( \widetilde{Q}))+n\log(\sigma)\lesssim \widehat{\log(\det(Q))} \lesssim (1+\epsilon)\mathrm{tr}(\log( \widetilde{Q}))+n\log(\sigma). 
   \label{eq_Hutc++ spd 2}
\end{equation}
\end{corollary}

\begin{proof}
By adding  $n\log(\sigma)$ in \eqref{eq_Hutc++ psd}, we obtain 
$$  (1-\epsilon)\mathrm{tr}(\log( \widetilde{Q})) + n\log(\sigma)\lesssim  \widehat{\mathrm{tr}( \log(\widetilde{Q}))} +n\log(\sigma)\lesssim  (1+\epsilon)\mathrm{tr}(\log( \widetilde{Q}))+n\log(\sigma),$$
then 
$$(1-\epsilon)\mathrel{tr}(\log( \widetilde{Q})) + n\log(\sigma)\lesssim \widehat{\log(\det(Q))}\lesssim  (1+\epsilon)\mathrm{tr}(\log( \widetilde{Q}))+n\log(\sigma).$$
\end{proof}
Next,  we derive error bounds of Log-determinant estimation by the  following proposition.
\begin{proposition}
\label{prop_leja_hutch++}
Let $\widetilde{Q}$ be defined as in \eqref{scaled_sigma_matrix} be SPD with $\sigma(\widetilde{Q})\subset [\alpha,\beta], 1\le \alpha < \beta$. Let $P_m$ denotes the $m$-degree  Newton - Léja polynomial interpolation.  \\
Assume that $P_m(x)\ge 0$ for every $x\in \sigma(\widetilde Q)$. Define 
\begin{equation*}
  L =\mathrm{tr}(\log(\widetilde{Q})),   \quad L_m=\mathrm{tr}(P_m(\widetilde{Q})).
\end{equation*}
 Let $\widehat{L}$, the Hutch++ trace estimation applied to $P_m(\widetilde{Q})$. Therefore for $\epsilon \in (0,1)$ and $\delta \in (0,1), \nu> 0 $ the trace error with probability at least $1-\delta$ satisfies 
\begin{equation}
    \vert \widehat{L}-L\vert \le n(1+m\mathrm{e}^{\nu m})C_\rho\rho^{-m} + \epsilon L_m,\qquad \rho>1.
\end{equation}
\end{proposition}
\begin{proof}
    Indeed $\widehat{L}= \widehat{\mathrm{tr}(P_m(\widetilde{Q}))}$,
we therefore  have 
$$ \widehat{L}-L= \widehat{L}-(L_m-\mathrm{tr}(E_m(\widetilde{Q}))),\quad \text{where}\quad E_m(x)=\log(x)-p(x). $$
This implies 
$$ \widehat{L}-L = \left(\widehat{L}-L_m\right)-\left(\mathrm{tr}(\log(\widetilde{Q})-P_m(\widetilde{Q}))\right).$$
Taking the absolute value and by Triangle inequality, we have
\begin{equation}
 \vert \widehat{L}-L \vert \le \vert \widehat{L}-L_m\vert + \vert \mathrm{tr}(\log(\widetilde{Q})-P_m(\widetilde{Q}))\vert. 
 \label{eq_tot_error_leja1}
\end{equation}
By \eqref{thm_hpp_psd}, we have $\vert \widehat{L}-L_m\vert \le \epsilon L_m$; and by spectral theorem  and \eqref{eq_lebesgue leja4}, we have $\vert \mathrm{tr}(\log(\widetilde{Q})-P_m(\widetilde{Q}))\vert \le n (1+\Lambda_{m+1})C_\rho \rho^{-m} $, thus \eqref{eq_tot_error_leja1} becomes
\begin{equation}
\vert \widehat{L}-L \vert \le n(1+\Lambda_{m+1})C_\rho \rho^{-m} + \epsilon L_m
    \label{eq_tot_error_leja2}
\end{equation}
For $\nu>0, \quad \Lambda_{m+1} \le m\mathrm{e}^{\nu m}$ as in \cite{totik2023lebesgue}. Therefore,  
we have with probability at least $1-\delta$
\[\vert \widehat{L}-L\vert \le n(1+m\mathrm{e}^{\nu m})C_\rho \rho^{-m}   + \epsilon L_m.
\]
\end{proof}

\Cref{prop_leja_hutch++} provides an error bound for Léja-Hutch++ when a fixed degree polynomial approximation is used and the corresponding polynomial matrix function is PSD. In practice, however, the Newton-Léja approximation is evaluated adaptively, and different matrix-vector products may require different degrees of interpolation. Consequently, the computed actions need not correspond to a single fixed polynomial matrix function. We therefore extend the analysis to inexact matrix logarithm actions, distinguishing the approximations errors introduced in the projected and residual components of the Hutch++ estimator.
\begin{theorem}[Error Bound of Inexact  Léja-Hutch++ log-det]
\label{thm_inexact_leja_hutchpp}
Let $Q\in \mathbb{R}^{n\times n}$ be SPD, and let $\widetilde{Q},P$ and $u_j$ define as in  \eqref{scaled_sigma_matrix}, \eqref{eq-proj_leja} and \eqref{eq_res_vectors} respectively. Set $F=\log(\widetilde{Q})$ and $A=[a_1,\dots,a_k]$ with $k=m_{vec}/3$ denote the orthonormal basis associated with $P$. Consider the inexact Newton-Léja-Hutch++ estimator 
\begin{equation}
    \widehat{\log\det(Q)}=n\log(\alpha)+ \sum_{i=1}^k a_i^\top\widetilde{F}_p(a_i) + \dfrac{1}{r}\sum_{i=j}^r u_j^\top\widetilde{F}_r(u_j), \qquad \text{with}\quad r = \dfrac{m_{vec}}{3},
\end{equation}
where $\widetilde{F}_p$  and $\widetilde{F}_r$ denote respectively  the approximate matrix-logarithm actions in the projected and residual components. Assume that
$$\Vert \widetilde{F}_p(a_i)-F(a_i) \Vert \le \tau_p, \qquad i = 1,\dots,k$$
and 
$$\Vert \widetilde{F}_r(u_j)-F(u_j) \Vert \le \tau_r\Vert u_j \Vert_2, \qquad j = 1,\dots,r.$$
Here $\tau_p \ge 0$ and $\tau_r \ge 0$ denote the prescribed  Newton-Léja approximation error bounds from the projected and residual matrix-logarithm actions, respectively. Then
\begin{equation}
    \vert \widehat{\log\det(Q)}-\log\det(Q)\vert \le \epsilon_{stoch} + k\tau_p + \dfrac{\tau_r}{r}\sum_{j=1}^r\Vert u_j\Vert^2_2,
    \label{eq_inexact_leja_hutch}
\end{equation}
where $$ \epsilon_{stoch}=\big \vert \dfrac{1}{r}\sum_{j=1}^ru_jFu_j - \mathrm{tr((I-P)F(I-P))}\big \vert.$$
\end{theorem}
\begin{proof}
Let $L_{exact}$ denote the corresponding estimator obtained using the exact matrix logarithm action while keeping the same projection subspace and residual vectors. By the trace decomposition and identity 
$\log\det(Q)=n\log(\alpha)+\mathrm{tr}(F),$ we have 
$$\vert L_{exact}-\log\det(Q)\vert=\epsilon{_{stoch}}.$$ 
Therefore, the triangle inequality yields
\begin{equation}
 \vert \widehat{\log\det(Q)} - \log\det(Q)\vert \le \vert \widehat{\log\det(Q)}-L_{exact}\vert + \epsilon_{stoch}.
 \label{eq_inexact_leja_pre}
\end{equation}

Using the definition of the two estimators, we obtain 
$$\vert \widehat{\log\det(Q)}-L_{exact}\vert \le \sum_{i=1}^k\vert a_i^\top (\widetilde{F}_p(a_i)-F(a_i))\vert + \dfrac{1}{r}\sum_{j=1}^r\vert u_j^\top (\widetilde{F}_p(u_j)-F(u_j))\vert.$$
Apply the Cauchy-Schwarz inequality, the orthogonality of $A$ and the assumed action error bounds, give 
$$\vert \widehat{\log\det(Q)}-L_{exact}\vert \le k\tau_p + \dfrac{\tau_r}{r}\sum_{j=1}^{r}\Vert u_j\Vert_2^2.$$
Combining the precedent inequality with \eqref{eq_inexact_leja_pre} yields \eqref{eq_inexact_leja_hutch}.
\end{proof}
\section{Numerical Experiments}
\label{experiments}
In this section, we compare the proposed Log-determinant method against a wide variety of sparse matrices. Our benchmark includes both realistic data sets from the University of Florida(UFL) Sparse Matrix Collection \cite{davis2011university} and  synthetic pentadiagonal matrices. Some of these UFL matrices were also used in \cite{han2015large,boutsidis2017randomized,ubaru2017fast}. The actual matrices occur over a range of important application areas such as  computational fluid dynamics simulations, finite element models in structural engineering, elasticity benchmarks, network admittance matrices used in power system analysis, material science, plasma physics, and modeling crystalline structures. Aside from these, we also generate pentadiagonal symmetric definite (SPD) matrices, which provide an easy but representative class of structured banded problems.

\quad
This broad spectrum of matrices allows us to experiment with our method under all situations: very sparse and unstructured meshes (finite element analysis, computational fluid dynamics), structured systems based on graphs (power grids), dense local coupling (molecular and crystal models), physics-based operators, and synthetic banded operators. Such variation is required because it evidences how  our Léja points-based interpolation  algorithm performs in practice under differing sparsity patterns and spectral distributions (see table \ref{table 1 leja}). Léja ( $s_{val} = \lambda_{max}/2$) and Léja ($s_{val} = c$) represent, respectively, our methods based, respectively, on the centering strategy from $s_{val} = \lambda_{max}/2$ and $s_{val} =c$. We set the  Léja adaptive tolerance to $10^{-6}$ in \Cref{algoLéja} and  Taylor degree $p=250$ in \Cref{algo:dd_stable}  in all our experiments.   All experiments in this work  are performed on a \textbf{MacBook Pro with Apple M1 Pro chip and 16 GB RAM, using Python 3.11.8}. The code of our method can be found here \footnote{\url{https://github.com/VerlonRoelMBINGUI/Log_determinant-using-Leja-Point.git}}

\begin{table}[htbp]

\caption{Description of sparse matrices used in the experiments. 
Real matrices are from the UF Sparse Matrix Collection \cite{davis2011university}.}
\label{table 1 leja}
\centering
\renewcommand{\arraystretch}{1.2}
\begin{tabular}{|l| l| c|}
\hline
Matrix        & Application Domain              & Size ($n \times n$) \\
\hline
\texttt{pdb1HYS}      &  protein structure (bioinformatics)  & $36,417 \times 36,417$ \\
\texttt{ship\_001}      &  ship section (structural engineering)    & $34,920 \times 34,920$ \\
\texttt{boneS01}      & Model Reduction  (biomedical engineering)      & $ 127,224 \times 127,224$ \\
\texttt{ecology2} &   Circuit theory applied to animal/gene flow      & $ 999,999
 \times 999,999
$ \\
\texttt{af\_shell8}  &  Structural mechanics and Sheet metal forming          & $ 504,855 \times 504,855 $ \\
\texttt{bone010}  & Biomechanics and Biomechanical engineering     & $986,703 \times 986,703 $ \\
\texttt{crystm01}  & Materials science (crystal)     & $4,300 \times 4,300$ \\

 \texttt{crystm02}  & Material science (crystallography)  &  $13,965 \times 13,965$\\
 \texttt{gyro\_k} & Plasma physics (fusion)   & $ 17,361 \times 17,361 $ \\
 \texttt{wathen120}  & 2D FEM elasticity benchmark  & $36,441 \times 36,441$ \\
\hline
\end{tabular}

\end{table}


\subsection{Log Determinant}
\label{sub_sec_exp_log_det_leja}
\paragraph{\textbf{Using Pentaadiagonal generated matrices}}
 We test the performance of our method on large, synthetically created pentadiagonal matrices.
The matrices were created independently sampling all of the diagonals with a uniform distribution in the interval $[0,1]$, using command \texttt{np.random.rand}, and assembling them using the \\ \texttt{scipy.sparse.diags} function in Python. In particular, the main diagonal, two first off-diagonals (offsets $\pm 1$), and two second off-diagonals (offsets $\pm 2$) were constructed as

\texttt{diagonals} = [ \texttt{np.random.rand(n)}, \texttt{np.random.rand(n-1)}, 

\texttt{np.random.rand(n-2)}, \texttt{np.random.rand(n-1)}, \texttt{np.random.rand(n-2)}],

and combined into a sparse CSR matrix with
$$Q = \texttt{sp.diags(diagonals, [0,1,2,-1,-2], format='csr')}.
$$
To guarantee positive definiteness, the matrix was symmetrized and shifted by an integer multiple of the identity,
$Q \; \leftarrow \; Q + Q^{\top} + nI,
$
with $n$ the matrix size. This ensures all eigenvalues are strictly positive, making the matrices appropriate for Log-determinant computations.
For our experiments, we used sizes of up to $10^7$, generating very large structured sparse systems.
Pentadiagonal matrices of this type arise naturally in scientific computing, for example, in the discretization of higher-order differential operators.
Fourth-order finite difference methods for biharmonic equations in elasticity (e.g., plate bending models) or fluid dynamics (e.g., streamfunction formulations of the Stokes problem) give pentadiagonal systems.
Likewise, high-order discretizations of one-dimensional beam vibration problems produce banded matrices with such structure.
Through experimentation on synthetically generated SPD pentadiagonal matrices produced thereby, we mimic the sparsity and conditioning characteristics of real applications while maintaining problem size and numerical properties under our control.\\

\begin{figure}
     \centering
 
         \includegraphics[width=0.45\textwidth]{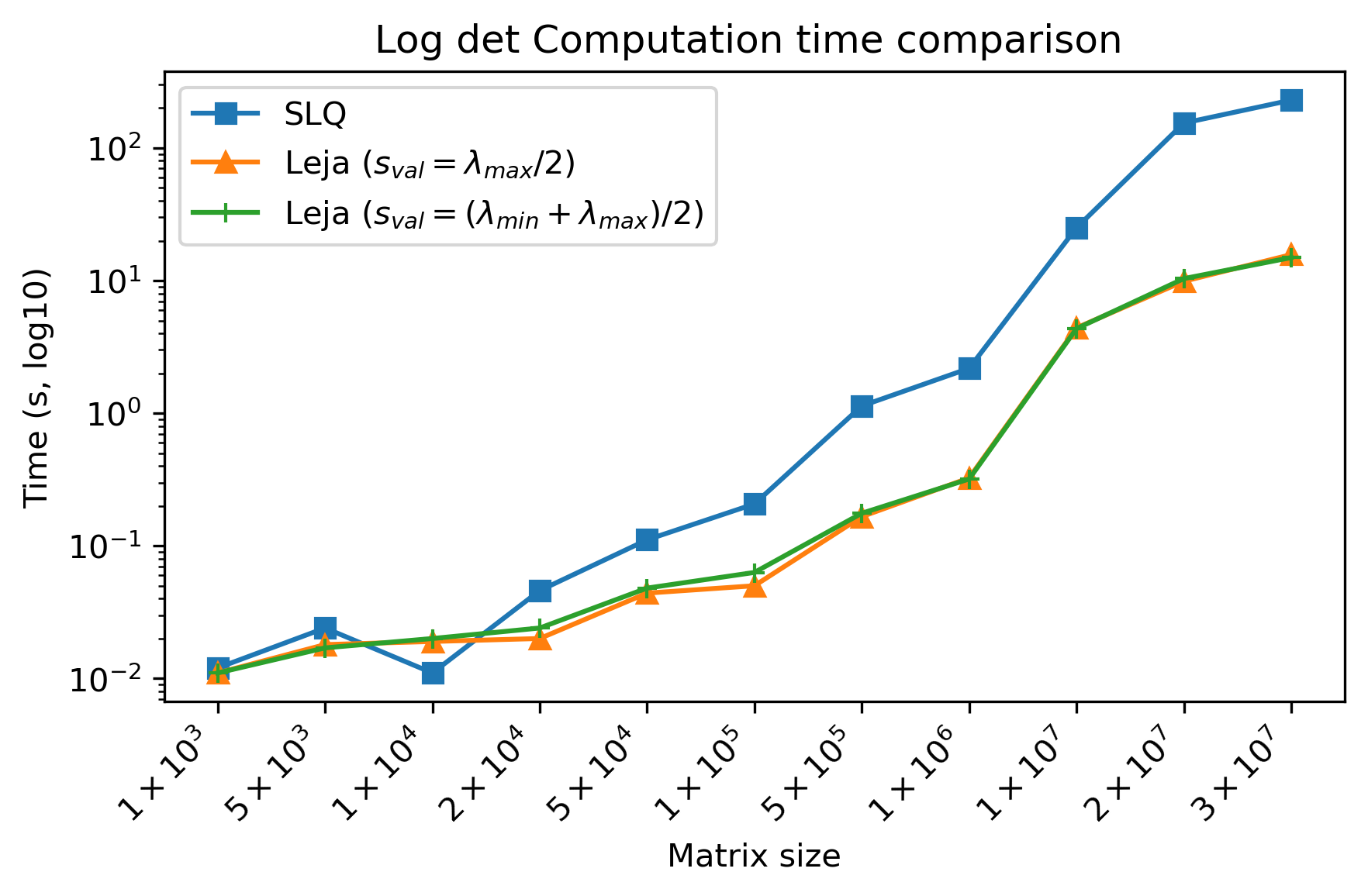}
         \includegraphics[width=0.45\textwidth]{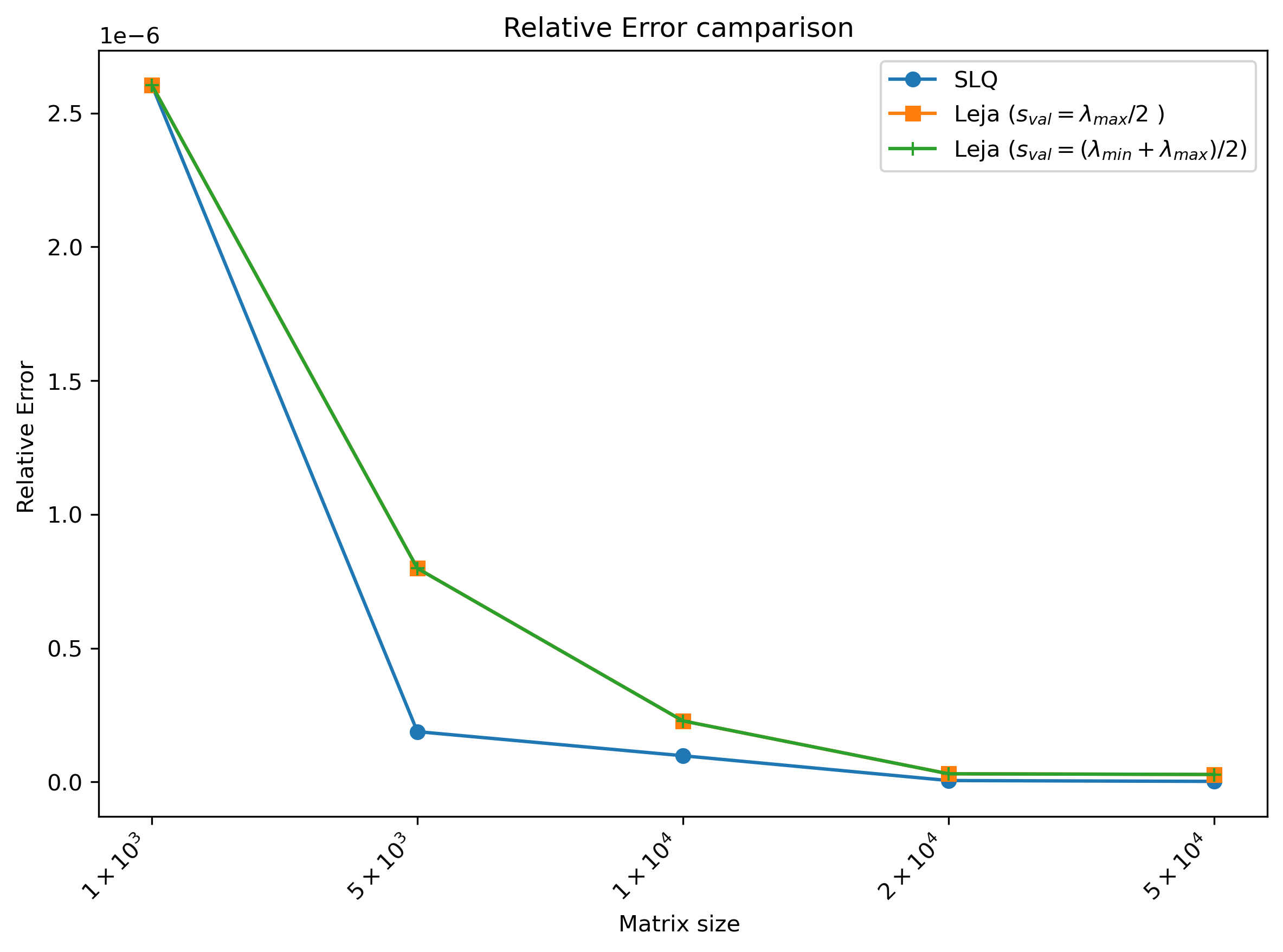}
     \caption{SLQ and Léja based techniques performance comparison on pentadiagonale matrices  accros variouss size. The matrix structure is defined  by $Q$ = sp.diags(diagonals, [0,1,2,-1,-2], format='csr').}
     \label{fig_total_comparison}
\end{figure}
Figure \ref{fig_total_comparison} shows that \emph{SLQ} typically has the smallest overall relative errors at such sizes, but its behavior is more erratic for intermediate sizes and much worse at the largest $n$. Our proposed method is consistently the fastest for all sizes with small relative errors typically. Overall, these results indicate that the proposed method offers the best \emph{speed} with good accuracy. Cholesky factorization is used as the exact the Log-determinant method although  it requires too much time for computing. 
\paragraph{\textbf{Using Real problems matrices for UFL}}
In the first experiment, we evaluated the performance of our method with respect to Cholesky decomposition and the Stochastic Lanczos Quadrature (SLQ) \cite{ubaru2017fast},  So-SLQ \cite{han2023suboptimal}, funNyström++ \cite{cortinovis2022randomized}  Krylov\_aware \cite{chen2023krylov}. Comparisons  in \cite [Fig. 1 and Table 3]{ubaru2017fast} show that SLQ is superior to Chebyshev. Note  that  A '-', symbol in  \Cref{tab:logdet_comparison_leja} indicates that the exact computation could not be
performed on our hardware due to memory breakdowns.
The estimated log-determinant  with their CPU times ( in seconds)  are respectively reported  in \Cref{tab:logdet_comparison_leja}. In all experiments, we used 12 probe vectors and set the adaptive tolerance to $ tol=10^{-6}$ in \Cref{algoLéja} and the Taylor degree$ p = 250 $ in \Cref{algo:dd_stable}. For each matrix, the proposed adaptive Léja-Hutch++ method is executed over 10 independent trials, and the mean number of matrix-vector products (matvecs) it requires is recorded. This number is then  used  as a matvec budget, i.e., the total number of matvecs  each  competing method is allowed to use on that matrix. To ensure a fair comparison, all methods are evaluated under the same matrix-specific matvec budget.  

\begin{table}[htbp]
\centering
\caption{Comparison of log-determinant estimates and
computational times under matrix-specific matvec budgets. Remember that, for fair comparison, all  methods in comparison are evaluated under the same matrix-specific matvec budget.}
\label{tab:logdet_comparison_leja}

\setlength{\tabcolsep}{3pt}
\renewcommand{\arraystretch}{2.15}

\resizebox{\textwidth}{!}{%
\begin{tabular}{l c c *{5}{cc}}
\toprule

\multirow{2}{*}{Matrix}
& \multirow{2}{*}{Cholesky}
& \multirow{2}{*}{Matvec Budget}
& \multicolumn{2}{c}{SLQ}
& \multicolumn{2}{c}{So-SLQ}
& \multicolumn{2}{c}{\textbf{Léja (Ours)}}
& \multicolumn{2}{c}{funNystr\"om++}
& \multicolumn{2}{c}{Krylov-aware} \\

\cmidrule(lr){4-5}
\cmidrule(lr){6-7}
\cmidrule(lr){8-9}
\cmidrule(lr){10-11}
\cmidrule(lr){12-13}

& & &
Estimate & CPU (s)
& Estimate & CPU (s)
& \textbf{Estimate} & \textbf{CPU (s)}
& Estimate & CPU (s)
& Estimate & CPU (s) \\

\midrule

\texttt{crystm02}    & -406912.286 & 3765& -406919.235 & 15.167 & -406940.923 & 1.329 & \textbf{-406920.662} &\textbf{1.028} & -406941.721 & 14.923 & -406859.441 & 55.473 \\
\texttt{gyro\_k}    & 304741.205& 876& 304964.804 & 3.086 & 305035.979 & 1.259 & \textbf{308236.544} & \textbf{1.193} & 304630.401 & 2.131 & 304984.424 & 1.514 \\
\texttt{wathen120}     & 127080.312 & 900 & 127093.592 & 4.638& 127094.725 & 0.956 & \textbf{127148.284} & \textbf{0.580} & 127021.401 & 1.626 & 127033.492 & 1.598 \\
\texttt{pdb1HYS}      & 170081.683  & 900 & 170247.713 & 7.534& 170227.952 & 3.451 & \textbf{170378.111} & \textbf{4.183} & 169911.079 & 5.607 & 170159.538 & 4.954 \\
\texttt{ship\_001}      & 761021.440 & 864 & 769221.899 & 6.951 & 770445.750 & 3.636 & \textbf{797307.587} & \textbf{4.042} & 762052.405 & 4.820 & 761717.540 & 4.698 \\
\texttt{crystm01}      & -137117.213 & 2448 & -137123.899 & 3.008 & -137126.686 & 0.450 & \textbf{-137114.801} & \textbf{0.381} & -137114.498 & 3.831 & -137111.011 & 12.376 \\
\texttt{boneS01}     & -- & 900&  1103944.763&14.805 &1103885.165 &5.199 & \textbf{1104717.756} & \textbf{4.409} & 1103896.421 & 8.702 & 1103257.385 &8.191 \\
\texttt{ecology2}     & --& 948  & 3394228.267 & 65.214 & 3394291.992 & 15.521& \textbf{3395827.209} & \textbf{8.931} & 3394157.954 & 44.723& 3392507.021 & 24.642 \\
\texttt{bone010}      & --& 948 & 9290564.168 & 97.190 & 9290747.910 & 37.414 & \textbf{9291384.945} & \textbf{32.079} & 9290482.289 & 61.022& 9288489.043 & 50.962 \\
\texttt{af\_shell8}   & -- & 900 & 6466366.472 & 45.392 & 6466418.870 & 14.970 & \textbf{6467671.298} &\textbf{ 11.063} & 6485430.965 & 23.298 & 6465847.227 & 21.551 \\

\bottomrule
\end{tabular}%
}
\end{table}

\begin{table}[htbp]
\caption{Relative error comparison between SLQ, So-SLQ, funNyström++, and Krylov-aware  with our proposed method over 10 trials for small and medium size matrices. The exact log-determinant is computed using Cholesky factorization.}
\label{table 4 leja}
\centering
\renewcommand{\arraystretch}{2}
\setlength{\tabcolsep}{2.5pt}      
\small                            
\resizebox{\textwidth}{!}{%
\begin{tabular}{|l| c|c|c| c|c|}
\hline
Matrices & SLQ Rel. err.& So-SLQ & Léja Rel. err.& funNyström++ Rel. err. & Krylov\_aware  \\
\hline
\texttt{crystm02}  & $1.708 \times 10^{-5}$ &$7.038\times 10^{-5}$&$2.058\times 10^{-5}$ &$7.234\times 10^{-5}$ &$1.299\times 10^{-4}$\\
\texttt{gyro\_k} & $7.337\times 10^{-4}$  &  $9.673\times 10^{-4}$& $ 1.147 \times 10^{-2}$ &$3.636 \times 10^{-4} $ &$ 7.981 \times 10^{-4}$\\
\texttt{wathen120}  & $1.045 \times 10^{-4}$  & $1.134\times 10^{-4}$&$5.349\times 10^{-4}$ &$4.636\times 10^{-4}$  &$3.684\times 10^{-4}$ \\
\texttt{pdb1HYS}  & $9.762 \times 10^{-4}$& $8.600 \times 10^{-4}$ &$1.743\times 10^{-3}$ &$1.003\times 10^{-3}$ &$4.578\times 10^{-4}$\\
\texttt{ship\_001}  & $1.078 \times 10^{-2}$& $1.238 \times 10^{-2}$  &$4.768 \times 10^{-2}$ &$1.355\times 10^{-3}$ &$9.147 \times 10^{-4}$\\
 \texttt{crystm01}  & $4.876 \times 10^{-5}$ & $6.909\times 10^{-5} $& $1.759\times 10^{-5} $ & $1.980\times 10^{-5} $ &$4.523\times 10^{-5} $ \\
\hline
\end{tabular}%
}
\end{table}

\begin{figure}[htbp]
  \centering
  \includegraphics[width=0.9\textwidth]{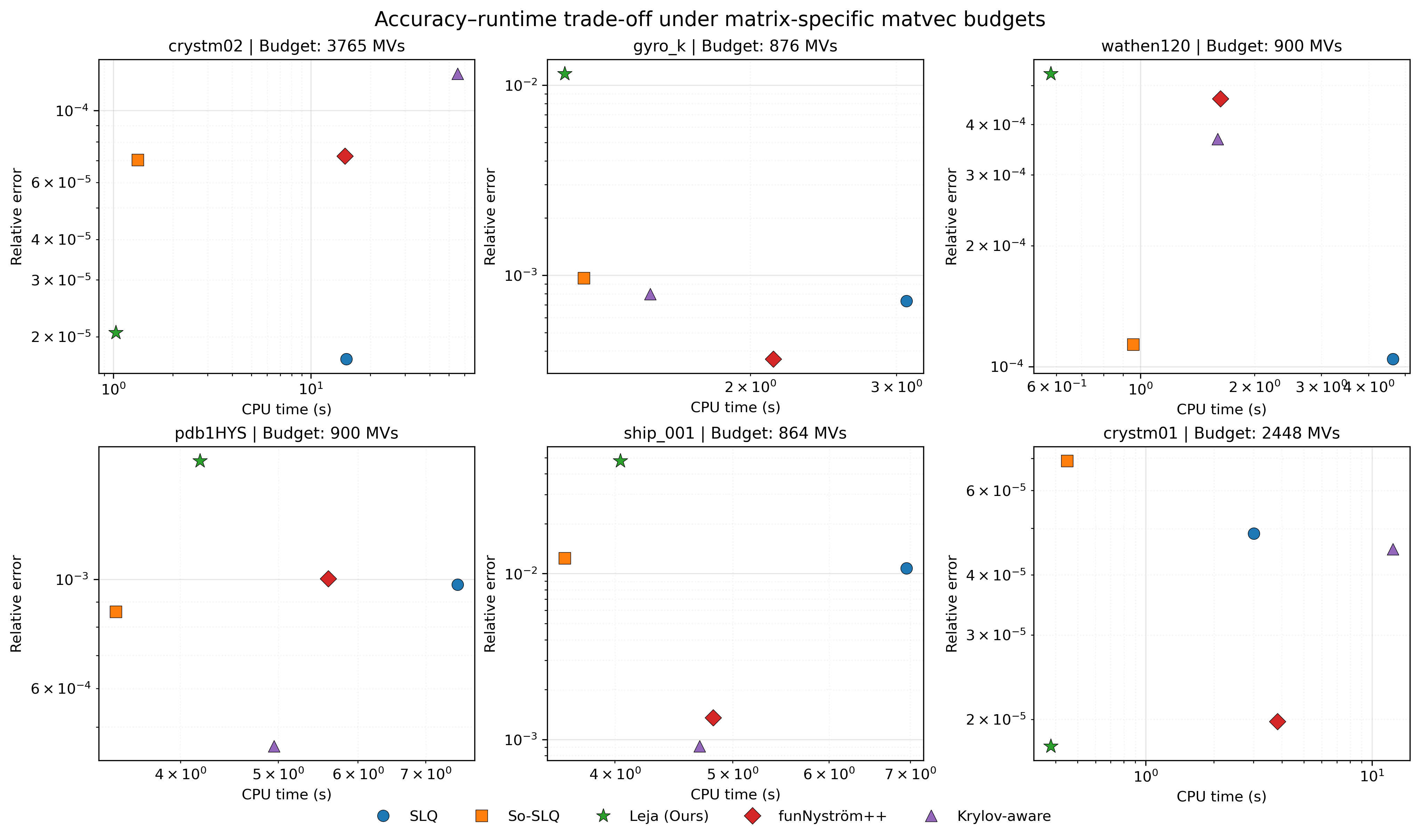}
  \caption{Accuracy-runtime trade-off of SLQ, So-SLQ,
    adaptive L\'eja-Hutch++, funNystr\"om++,
    and Krylov-aware across six small and medium size sparse SPD matrices.
    Each panel corresponds to a different matrix,
    with all methods evaluated under the same
    matrix-specific matvec budgets.}
  \label{fig_error_cpu_time_paper2_small}
\end{figure}
\newpage
\begin{table}[htbp]
\caption{Relative error comparison between SLQ, SO-SLQ, funNyström++, Krylov-aware, and our proposed method over 10 trials for large matrices. The reference log-determinant is computed using SLQ with a large matvec budget of 9000, averaged over 5 trials. Note that this high-budget SLQ serves only as the ground truth and is distinct from the equal-budget SLQ baseline reported in \Cref{tab:logdet_comparison_leja}. }
\label{table 5 leja}
\centering
\renewcommand{\arraystretch}{2}
\setlength{\tabcolsep}{2.5pt}      
\small                            
\resizebox{\textwidth}{!}{%
\begin{tabular}{|l| c|c|c| c|c|}
\hline
Matrices & SLQ Rel. err.& So-SLQ & Léja Rel. err.& funNyström++ Rel. err. & Krylov-aware  \\
\hline
\texttt{boneS01}  & $5.589 \times 10^{-5}$ &$1.099\times 10^{-4}$&$6.443\times 10^{-4}$ &$9.968\times 10^{-5}$ &$6.785\times 10^{-4}$\\
\texttt{ecology2} & $1.352\times 10^{-4}$  &  $1.164\times 10^{-4}$& $3.358 \times 10^{-4}$ &$1.559 \times 10^{-4} $ &$ 6.422\times 10^{-4}$\\
\texttt{bone010 }  & $5.437 \times 10^{-6}$  & $1.434\times 10^{-5}$&$8.291\times 10^{-5}$ &$1.425\times 10^{-5}$  &$2.288\times 10^{-4}$ \\
\texttt{af\_shell8}  & $5.043 \times 10^{-5}$& $5.853 \times 10^{-5}$ &$2.522\times 10^{-4}$ &$2.999\times 10^{-3}$ &$2.987\times 10^{-5}$\\
\hline
\end{tabular}%
}
\end{table}
\begin{figure}[htbp]
  \centering
  \includegraphics[width=0.7\textwidth]{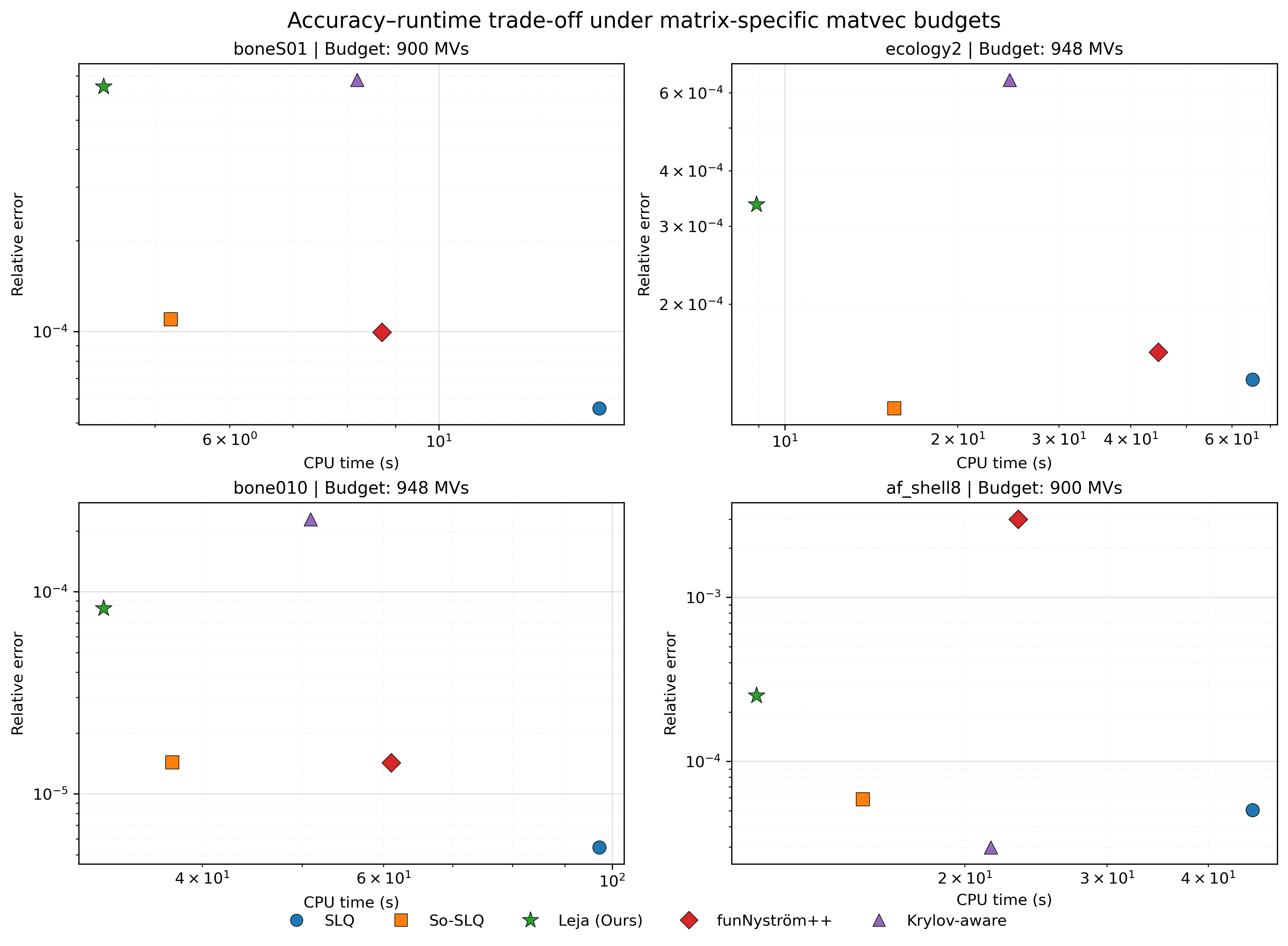}
  \caption{Accuracy-runtime trade-off of SLQ, So-SLQ, adaptive L\'eja-Hutch++, funNystr\"om++, and Krylov-aware across four large sparse SPD matrices.
Each panel corresponds to a different matrix, with all methods evaluated under the same matrix-specific matvec budgets  over 10 trials.}
  \label{fig_error_cpu_time_paper2_large}
\end{figure}

The accuracy-runtime comparisons across ten large SPD matrices demonstrate that the proposed Léja-Hutch++ method method achieves substantial computational efficiency while remaining competitive accuracy across a range of matrix structures and special characteristics. The method attains the lowest CPU time on eight of the ten matrices, with notable speedups over standard SLQ on \texttt{crystm02}($14\times $),\\ \texttt{ecology2}($7.3 \times$), \texttt{bone010} ($3 \times$) (see \Cref{tab:logdet_comparison_leja}). Particularly promising results are observed on \texttt{crystm01}, where the proposed method achieves both the smallest relative error $1.759\times10^{-5}$ and the shortest computation time $0.381 s$. On \texttt{crystm02},  \texttt{boneS01}, and  \texttt{ecology2}, the method also delivers competitive accuracy with significantly reduced computational costs. Although SLQ, So-SLQ, and Krylov\_aware attain smaller relative errors on certain matrices \Cref{table 4 leja}, \Cref{table 5 leja}, the proposed method  generally offers a favorable trade-off between accuracy and computational efficiency (\Cref{fig_error_cpu_time_paper2_small} and \Cref{fig_error_cpu_time_paper2_large}), especially for large matrices (\Cref{fig_error_cpu_time_paper2_large}). These findings highlight the computational advantages of our method and demonstrate its potential as an efficient and scalable alternative for large-scale log-determinant estimation.
\newpage
\subsection{Independent assessment of  Matrix Logarithm Actions}
We compare Newton-Léja interpolation; Chebyshev interpolation and the Lanczos method to approximate $\log(Q)v$, using ten independent random vectors and reporting the mean relative error against the number of matrix-vector products. The experiments, consider the pentadiagonal matrix described in \Cref{sub_sec_exp_log_det_leja} with $n=50000$ shows that all three methods converge rapidly, with Newton-Léja and Chebyshev  reaching relative errors near machine precision after approximating 4 matrix-vector products. 
\begin{figure}[htbp]
  \centering
  \includegraphics[width=0.45\textwidth]{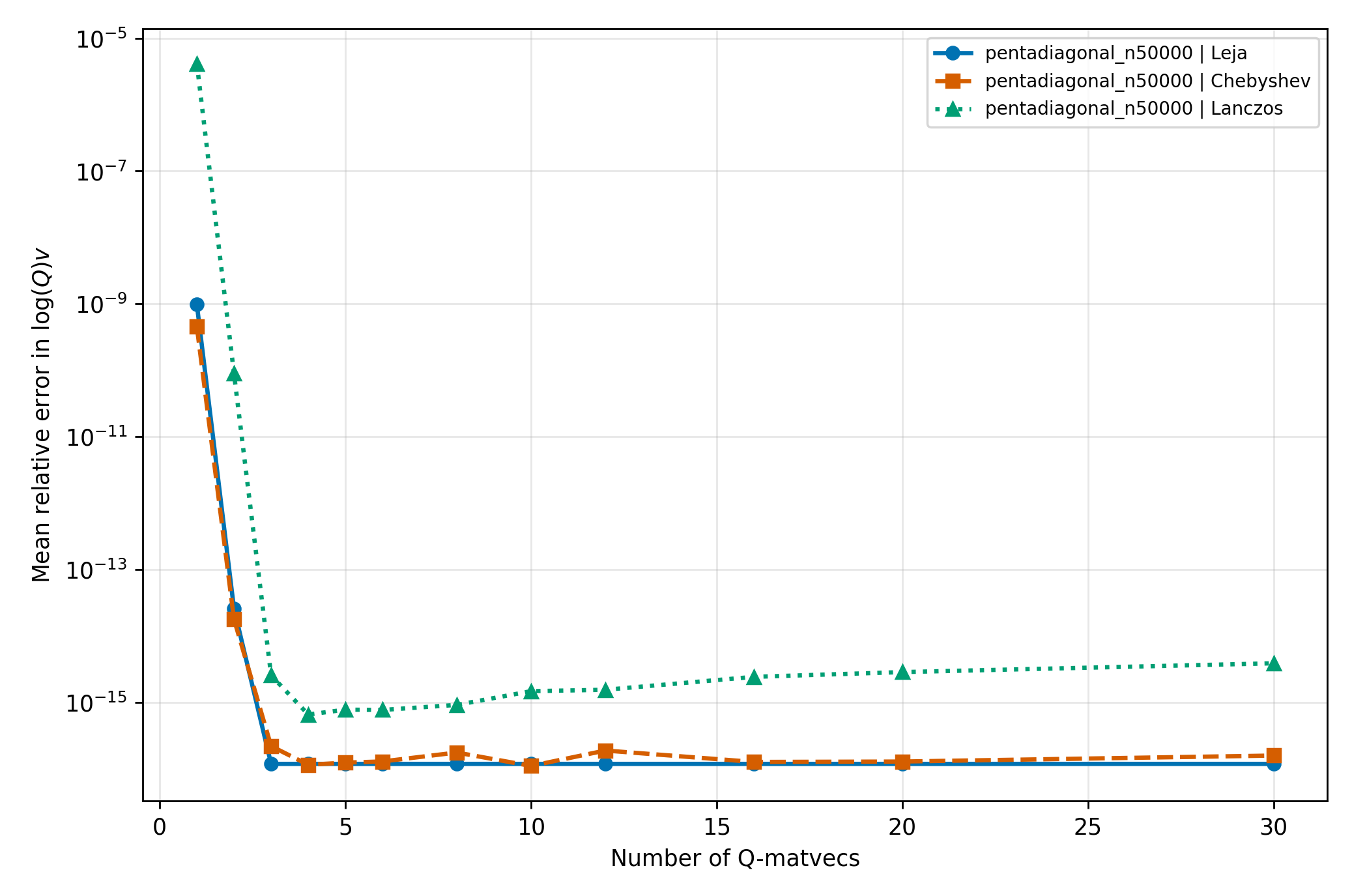}
  \caption{Independent comparison of matrix logarithm action approximations }
  \label{fig_log_leja_action}
\end{figure}

\subsection{Numerical Validation of the Inexact-Oracle Error Bound}
To assess the inexact-oracle error bound established in \Cref{thm_inexact_leja_hutchpp},  we conduct numerical experiments on two sparses SPD matrices \texttt{wathen120} and \texttt{crystm02}. For each matrix, ten independent realization are performed. Fully reorthogonalized Lanczos approximations are used as numerical references for the matrix-logarithm actions, while Cholesky factorization provides the reference log-determinants, respectively. The action-error parameters $\tau_p$ and $\tau_r$ are estimated from the maximum projected and normalized residual action errors, respectively, using the Lanczos references.\\
For  both matrices, \Cref{fig_log_leja_inexact_leja} shows that, the observed errors remain below their respective bounds across all ten realizations, supporting the numerical consistency of the theoretical error decompisition.
\begin{figure}[htbp]
  \centering
  \includegraphics[width=0.5\textwidth]{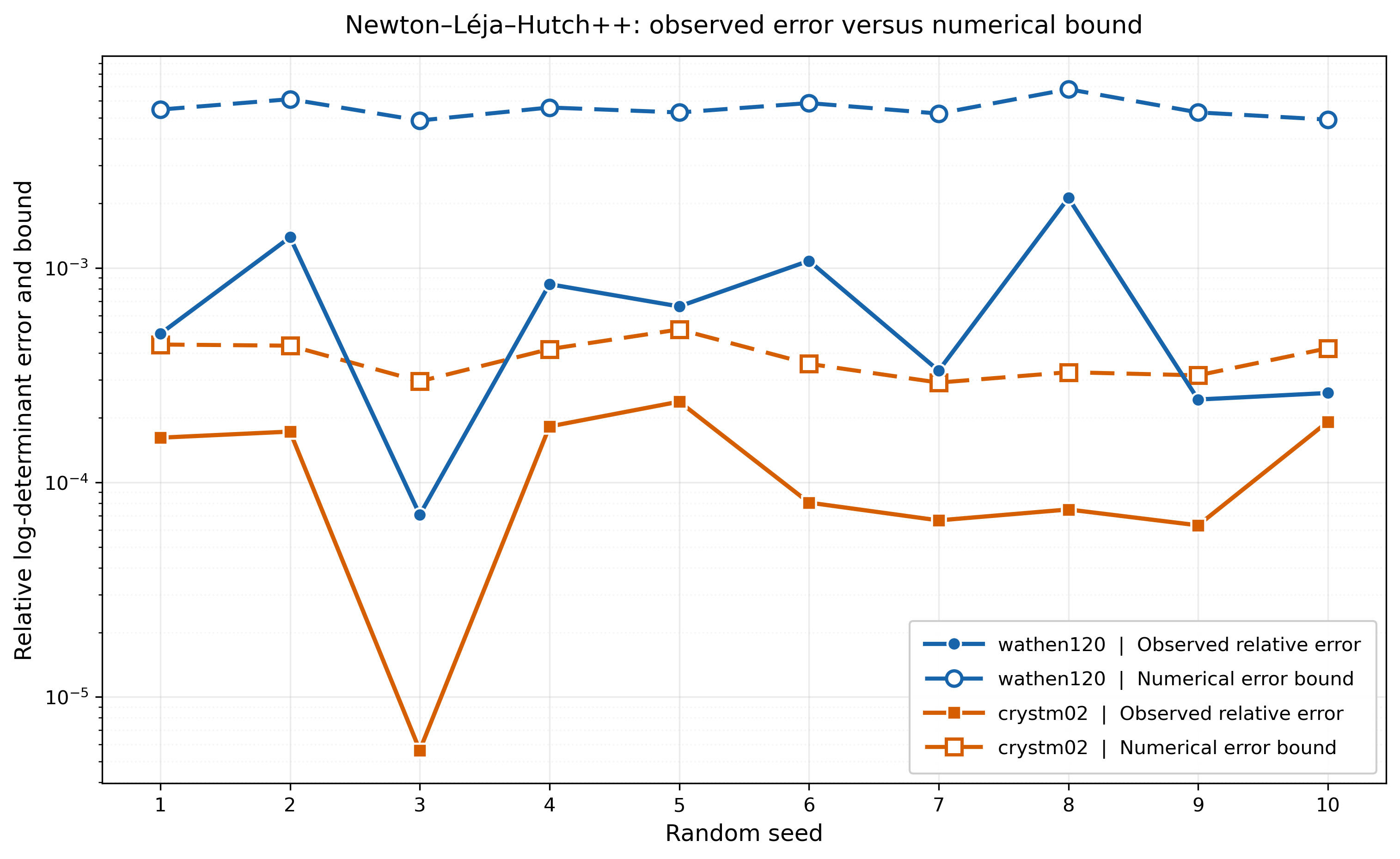}
  \caption{Numerical validation of the inexact-oracle error bound for Léja-Hutch++ on \texttt{wathen120}\texttt{crystm02}  over ten independent random realizations. Observed relative log-determinant errors are compared with the corresponding numerically evaluated error bounds.}
  \label{fig_log_leja_inexact_leja}
\end{figure}
\subsection{Maximum likelihood estimation for GMRF}
In this section, we apply our novel proposed method of approximating the 
log determinants to evaluate the maximum likelihood(ML) in Gaussian Markov Random fields.
\paragraph{\textbf{Gaussian Markov Random fields (GMRF) with 100 million variables on synthetic data.}}
As in Section 5.2 of \cite{han2017approximating}, Our synthetic data is generated from a $2D$  GMRF on a square grid of size $5000 \times 5000$ with a precision matrix $Q \in \mathbb{R}^{n\times n}$ $n = 25\times 10^{6}$, with four nearest neighbor structures and partial correction $\theta = -0.22$. A Sample $x$ from the GRMF model is generated using a Gibbs sampler for parameter $\theta$. The likelihood of this sample is then given by 
$$ 2\log(p(x/\theta))=\log(\det(Q(\theta)))-x^\top Q(\theta)x -n\log(2\pi), $$
with $Q(\theta)$ is a matrix of dimension $25\times 10^{6}$ and approximately $10^{8}$ non-zero entries. Therefore, it is required to solve: 
$$ \max_{\theta}\log(p(x/\theta)).$$
In the previous equation, estimating $\log(\det(Q(\theta))$ is the main challenge, especially when $n$ is very large.
Four $4$- neighbor grids, the eigen value of $Q(\theta)$,
$$ \lambda(Q(\theta)) \subset [1-4\vert\theta|, 1+4\vert\theta|]. $$
Then, we define $\sigma = 1-4\vert\theta|$  and then rewrite $$ \mathrm{tr}\widehat{(\log(Q(\theta)))}=\widehat{\mathrm{tr}\log\big(\frac{1}{\sigma}Q(\theta)\big)} + n\log(\sigma).$$
Where the eigen values of  $\frac{1}{\sigma}Q(\theta)$ are greater than equal to $1$ and  $\log\big(\frac{1}{\sigma}Q(\theta)\big)$ PSD. Here $\widehat{\mathrm{tr}\log\big(\frac{1}{\sigma}Q(\theta)\big)}$ our developed method. 
Figure \ref{fig_log_likelihood} confirmed that the log-likelihood is maximized at the correct hidden parameter $\theta$.

\begin{figure}[htbp]
  \centering
  \includegraphics[width=0.5\textwidth]{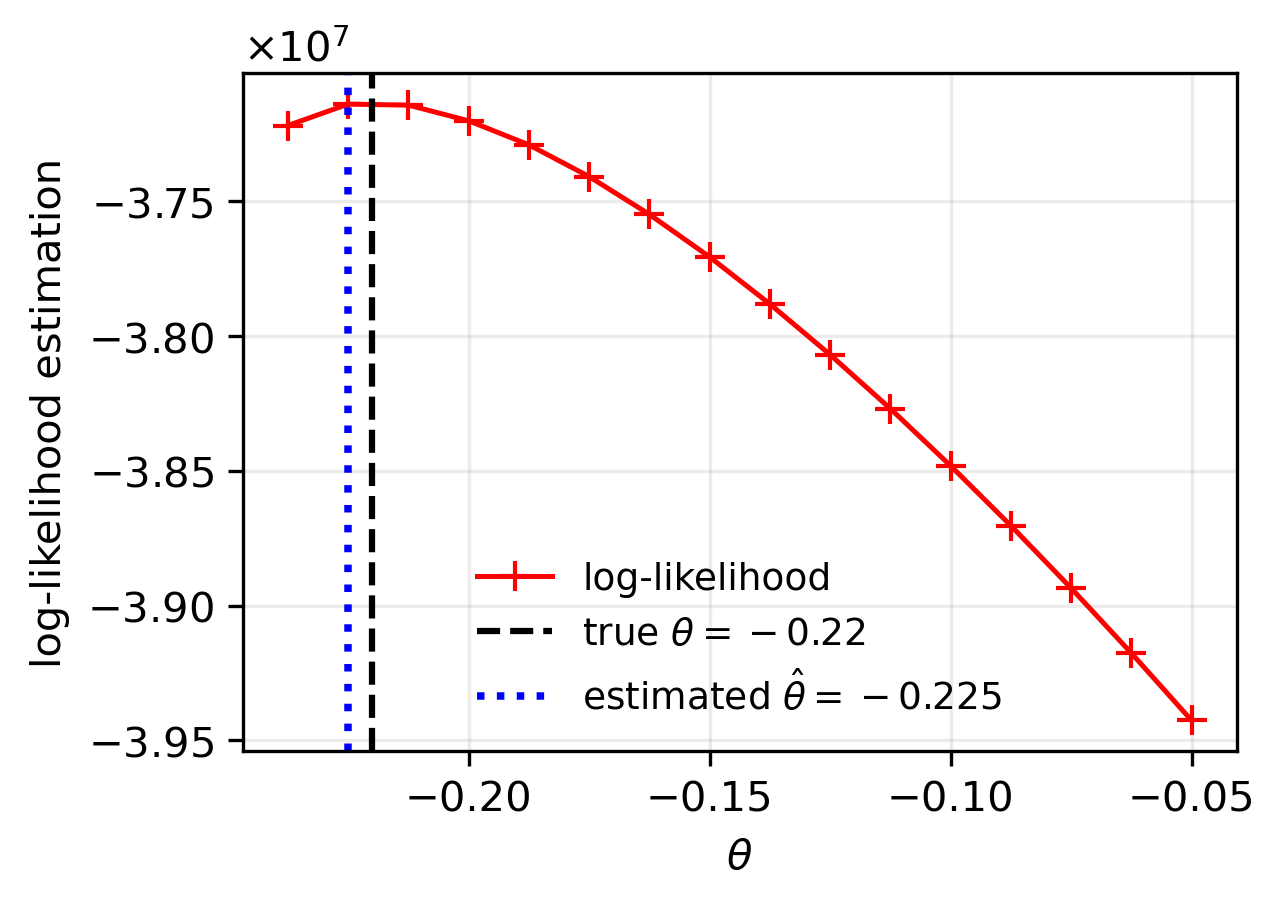}
  \caption{Log-likelihood estimation for hidden parameter $\theta$ for square GMRF model of size $5000\times 5000$.}
  \label{fig_log_likelihood}
\end{figure}

\section{Conclusions}
\label{sec:conclusions}
In this paper, we studied a novel inexpensive technique of estimating the Log-determinant of a large  sparse symmetric positive definite matrix by combining Léja points interpolation with the hutch++  stochastic trace estimator. The required requirement proposed requires only matrix vector products and coarse spectral information and incorporates an adaptive stopping strategy to determine the degree of interpolation. We also introduced a centering strategy for the compuatation of logarithm divided differences and  derived approximation error bounds for both  Léja points interpolation and the resulting log-determinant estimator. Numerical experiments on large real and synthetic sparse matrices demonstrate that the proposed Léja method can substantially reduce computational time compared with SLQ, So-SLQ, funNyström++ and Krylov-aware, while maintaining competitive accuracy on a range of test problems.  The results therefore highlight a favorable accuracy-cost trade-off and support the use of the Léja method as a practical alternative for large-scale matrix-free log-determinant estimation.

\appendix

\section{ Sensitivity Analysis of the Method Performance to Minimum and Maximum Computational Budgets } 
\label{appendix A}
To complement the mean-budget comparison presented above in the main manuscript, we investigate the sensitivity of the competing methods to the minimum and maximum matvec products observed for the proposed Léja-Hutch++. These budgets arise naturally from the adaptive selection of the Newton-Léja interpolation degree. By evaluating SLQ, So-SLQ, funNyström++, and Krylov-aware at these two  extreme budgets, we assess whether their relative accuracy and computational efficiency remain consistent across the computational range determined by Léja-Hutch++. \\
The findings in Tables $5$ to $10$ reveal that an increase in the computational budget does not necessarily improve the estimation accuracy while the CPU time increases. The above observation is especially true in case of large matrices. In \texttt{ecology2}, the increase in the budget from $768$ to $1044$ matvec products leads to an increase in the CPU times of SLQ from $47.34$ to $74.91$ seconds, So-SLQ from $13.17$ to $17.62$ seconds, funNyström++ from $18.23$ to $63.88$ seconds, and Krylov-aware from $21.01$ to $38.07$ seconds. In a similar vein, in \texttt{bone010}, the CPU time increases in SLQ from $65.55$ to $110.56$ seconds and in funNyström++ from $38.28 $ to $87.51$ seconds. However, in this regard, the improvement in accuracy is not consistent.\\
Taken together with the mean matvec products budget results given in the main body of the manuscript, this illustrates the value of an adaptive computational budget in Léja-Hutch++. Our approach adaptively tunes its interpolation effort rather than using a pre-defined fixed budget, which shows promise on several test matrices with good accuracy-runtime trade-off properties, especially on \texttt{crystm02}, \texttt{wathen120}, \texttt{crystm01}, \texttt{boneS01}, and \texttt{ecology2}. This is due to the relatively low amount of CPU time used by this approach, but this is not always the case for accuracy. Overall, the sensitivity analysis demonstrates that increasing the computational budget does not necessarily yield proportional accuracy gains, providing further motivation for adaptive computational strategies in large-scale log-determinant estimation.

\begin{table}[htbp]
\centering
\caption{Comparison of log-determinant estimates and
computational times under minimum matrix-specific matvec budgets.  Remember that, for fair comparison, all  methods in comparison are evaluated under the same matrix-specific matvec budget.}
\label{tab:logdet_comparison_leja_min_budget}

\setlength{\tabcolsep}{3pt}
\renewcommand{\arraystretch}{1.8}

\resizebox{\textwidth}{!}{%
\begin{tabular}{l  c *{4}{cc}}
\toprule

\multirow{2}{*}{Matrix}
& \multirow{2}{*}{Min Matvec Budget}
& \multicolumn{2}{c}{SLQ}
& \multicolumn{2}{c}{So-SLQ}
& \multicolumn{2}{c}{funNystr\"om++}
& \multicolumn{2}{c}{Krylov-aware} \\

\cmidrule(lr){3-4}
\cmidrule(lr){5-6}
\cmidrule(lr){7-8}
\cmidrule(lr){9-10}

& & 
Estimate & CPU (s)
& Estimate & CPU (s)
& Estimate & CPU (s)
& Estimate & CPU (s) \\

\midrule

\texttt{crystm02}    & 3765& -406919.235 & 15.167 & -406940.923 & 1.329& -406941.721 & 14.923 & -406859.441 & 55.473 \\
\texttt{gyro\_k}   & 756& 304997.175 & 2.237 & 305119.816& 1.021 & 304659.747 & 1.836 & 304998.704 & 1.446 \\
\texttt{wathen120}    & 756 & 127093.673 & 3.543& 127095.320 & 0.553  & 126964.042 & 0.718 & 127034.720 & 0.867 \\
\texttt{pdb1HYS}       & 756 & 170255.212 & 5.930 & 170243.999 & 2.610 &  169911.079& 6.086 & 170159.530 & 5.609 \\
\texttt{ship\_001}      & 708 & 770084.639 & 5.826 & 772584.942 & 2.796  & 761680.869 & 4.180 & 762205.811 & 3.861 \\
\texttt{crystm01}  & 2364 & -137123.108 & 3.033 & -137126.686 & 0.536 &  -137107.828 & 4.472 & -137100.031 & 12.091 \\
\texttt{boneS01}   & 756&  1103952.595&11.381 &1103910.428 &4.058  & 1103631.441 & 5.921 & 1103264.648 &5.924 \\
\texttt{ecology2}& 768  & 3394270.393 & 47.335 & 3394407.710 & 13.169&  3395772.485 & 18.230& 3392571.485 & 21.012 \\
\texttt{bone010} & 756 & 9290593.534 & 65.546 & 9290822.382 & 29.661 & 9292207.019 & 38.284& 9288531.775 & 46.587 \\
\texttt{af\_shell8} & 756 & 6466382.745 & 34.026 & 6466468.307 & 12.612 & 6466035.934 & 15.447 & 6465871.477 & 19.651 \\

\bottomrule
\end{tabular}%
}
\end{table}
{ \small{
\begin{table}[htbp]
\caption{Relative error comparison between SLQ, So-SLQ, funNyström++, and Krylov-aware  with our proposed method over 10 trials for minimum matvec budget for small and medium size matrices. The exact log-determinant is computed using Cholesky factorization.}
\label{table Rel_error_leja_min_budget_small}
\centering
\renewcommand{\arraystretch}{1.3}
\setlength{\tabcolsep}{1.5pt}      
\small                            
\resizebox{\textwidth}{!}{%
\begin{tabular}{|l| c|c|c| c|}
\hline
Matrices & SLQ Rel. err.& So-SLQ &  funNyström++ Rel. err. & Krylov\_aware  \\
\hline
\texttt{crystm02}  & $1.708 \times 10^{-5}$ &$7.038\times 10^{-5}$ &$7.234\times 10^{-5}$ &$1.299\times 10^{-4}$\\
\texttt{gyro\_k} & $8.400\times 10^{-4}$  &  $1.242\times 10^{-3}$&$2.673 \times 10^{-4} $ &$ 8.450 \times 10^{-4}$\\
\texttt{wathen120}  & $1.051 \times 10^{-4}$  & $1.181\times 10^{-4}$ &$9.149\times 10^{-4}$  &$3.588\times 10^{-4}$ \\
\texttt{pdb1HYS}  & $1.020 \times 10^{-3}$& $9.543 \times 10^{-4}$ &$1.003\times 10^{-3}$ &$4.577\times 10^{-4}$\\
\texttt{ship\_001}  & $1.191 \times 10^{-2}$& $1.519 \times 10^{-2}$  &$8.665\times 10^{-4}$ &$1.556 \times 10^{-3}$\\
 \texttt{crystm01}  & $4.299 \times 10^{-5}$ & $6.909\times 10^{-5} $ & $6.845\times 10^{-5} $ &$1.253\times 10^{-4} $ \\
\hline
\end{tabular}%
}
\end{table}}}
{ \small {
\begin{table}[htbp]
\caption{Relative error comparison between SLQ, So-SLQ, funNyström++, and Krylov-aware  with our proposed method over 10 trials for minimum matvec budget for large matrices. The reference log-determinant is computed using SLQ with a large matvec budget of 9000, averaged over 5 trials. Note that this high-budget SLQ serves only as the ground truth and is distinct from the SLQ baseline reported in  \Cref{tab:logdet_comparison_leja_min_budget}.}
\label{table Rel_error_leja_min_budget_large}
\centering
\renewcommand{\arraystretch}{1.3}
\setlength{\tabcolsep}{1.5pt}      
\small                            
\resizebox{\textwidth}{!}{%
\begin{tabular}{|l| c|c|c| c|}
\hline
Matrices & SLQ Rel. err.& So-SLQ &  funNyström++ Rel. err. & Krylov\_aware  \\
\hline
\texttt{boneS01}  & $4.880 \times 10^{-5}$ &$8.699\times 10^{-5}$ &$3.397\times 10^{-4}$ &$6.719\times 10^{-4}$\\
\texttt{ecology2} & $1.228\times 10^{-4}$  &  $8.230\times 10^{-5}$&$3.197 \times 10^{-4} $ &$ 6.232 \times 10^{-4}$\\
\texttt{bone010}  & $2.277 \times 10^{-6}$  & $2.236\times 10^{-5}$ &$1.714\times 10^{-4}$  &$2.242\times 10^{-4}$ \\
\texttt{af\_shell8}  & $5.295 \times 10^{-5}$& $6.618 \times 10^{-5}$ &$6.902\times 10^{-7}$ &$2.612\times 10^{-5}$\\
\hline
\end{tabular}%
}
\end{table}}}
{ \small {
\begin{table}[htbp]
\centering
\caption{Comparison of log-determinant estimates and
computational times under maximum matrix-specific matvec budgets.  Remember that, for fair comparison, all  methods in comparison are evaluated under the same matrix-specific matvec budget.}
\label{tab:logdet_comparison_leja_max_budget}

\setlength{\tabcolsep}{3pt}
\renewcommand{\arraystretch}{1.8}

\resizebox{\textwidth}{!}{%
\begin{tabular}{l  c *{4}{cc}}
\toprule

\multirow{2}{*}{Matrix}
& \multirow{2}{*}{Max Matvec Budget}
& \multicolumn{2}{c}{SLQ}
& \multicolumn{2}{c}{So-SLQ}
& \multicolumn{2}{c}{funNystr\"om++}
& \multicolumn{2}{c}{Krylov-aware} \\

\cmidrule(lr){3-4}
\cmidrule(lr){5-6}
\cmidrule(lr){7-8}
\cmidrule(lr){9-10}

& & 
Estimate & CPU (s)
& Estimate & CPU (s)
& Estimate & CPU (s)
& Estimate & CPU (s) \\

\midrule

\texttt{crystm02}    & 3765& -406919.235 & 15.167 & -406940.923 & 1.329& -406941.721 & 14.923 & -406859.441 & 55.473 \\
\texttt{gyro\_k}   & 936& 304951.884 & 3.706 & 305007.599& 1.361 & 304893.119& 1.810 & 304974.624 & 1.672 \\
\texttt{wathen120}    & 972 & 127093.673 & 3.543& 127095.320 & 0.553  & 126964.042 & 0.718 & 127034.720 & 0.867 \\
\texttt{pdb1HYS}       & 972 & 170255.579 & 6.301 & 170242.187 & 3.567 &  169891.987& 7.339 & 170158.043 & 6.891 \\
\texttt{ship\_001}      & 936 & 768901.221 & 7.839 & 769930.423 & 3.383  & 761956.200 & 6.192 & 761909.560 & 5.510 \\
\texttt{crystm01}  & 2484 & -137123.108 & 3.071 & -137126.686 & 0.791 &  -137093.451 & 4.481 & -137112.512 & 13.868 \\
\texttt{boneS01}   & 972&  1103942.595&16.550 &1103879.342 &5.514  & 1103058.334 & 11.742 & 1103253.481 &8.992 \\
\texttt{ecology2}& 1044  & 3394214.088 & 74.911 & 3394263.768 & 17.618&  3393967.044 & 63.876& 3392504.894 & 38.072 \\
\texttt{bone010} & 1044 & 9290554.014 & 110.559 & 9290729.229 & 40.980 & 9290692.152 & 87.513& 9288481.055 & 76.960 \\
\texttt{af\_shell8} & 972 & 6466360.367 & 50.240 & 6466405.261 & 16.195 & 6465953.369 & 33.077 & 6465845.448 & 26.062 \\

\bottomrule
\end{tabular}%
}

\end{table}}}

{ \small {
\begin{table}[htbp]
\caption{Relative error comparison between SLQ, So-SLQ, funNyström++, and Krylov-aware  with our proposed method over 10 trials for maximum matvec budget for large matrices. The exact log-determinant is computed using Cholesky factorization.}
\label{table Rel_error_leja_max_budget_small}
\centering
\renewcommand{\arraystretch}{1.3}
\setlength{\tabcolsep}{1.3pt}      
\small                            
\resizebox{\textwidth}{!}{%
\begin{tabular}{|l| c|c|c| c|}
\hline
Matrices & SLQ Rel. err.& So-SLQ &  funNyström++ Rel. err. & Krylov\_aware  \\
\hline
\texttt{crystm02}  & $1.708 \times 10^{-5}$ &$7.038\times 10^{-5}$ &$7.234\times 10^{-5}$ &$1.299\times 10^{-4}$\\
\texttt{gyro\_k} & $6.913\times 10^{-4}$  &  $8.742\times 10^{-4}$&$4.985 \times 10^{-4} $ &$ 7.660 \times 10^{-4}$\\
\texttt{wathen120}  & $1.051 \times 10^{-4}$  & $1.181\times 10^{-4}$ &$9.149\times 10^{-4}$  &$3.588\times 10^{-4}$ \\
\texttt{pdb1HYS}  & $1.022 \times 10^{-3}$& $9.437\times 10^{-4}$ &$1.115\times 10^{-3}$ &$4.490\times 10^{-4}$\\
\texttt{ship\_001}  & $1.035 \times 10^{-2}$& $1.171 \times 10^{-2}$  &$1.228\times 10^{-3}$ &$1.167 \times 10^{-3}$\\
 \texttt{crystm01}  & $4.299 \times 10^{-5}$ & $6.909\times 10^{-5} $ & $1.733\times 10^{-4} $ &$3.428\times 10^{-5} $ \\
\hline
\end{tabular}%
}
\end{table}}}

\begin{table}[htbp]
\caption{Relative error comparison between SLQ, So-SLQ, funNyström++, and Krylov-aware  with our proposed method over 10 trials for maximum matvec budget for large matrices. The reference log-determinant is computed using SLQ with a large matvec budget of $9000$, averaged over 5 trials. Note that this high-budget SLQ serves only as the ground truth and is distinct from the  SLQ baseline reported in \Cref{tab:logdet_comparison_leja_max_budget}.}
\label{table Rel_error_leja_max_budget_large}
\centering
\renewcommand{\arraystretch}{1.3}
\setlength{\tabcolsep}{1.3pt}      
\small                            
\resizebox{\textwidth}{!}{%
\begin{tabular}{|l| c|c|c| c|}
\hline
Matrices & SLQ Rel. err.& So-SLQ &  funNyström++ Rel. err. & Krylov\_aware  \\
\hline
\texttt{boneS01}  & $5.786 \times 10^{-5}$ &$1.152\times 10^{-4}$ &$8.588\times 10^{-5}$ &$6.821\times 10^{-4}$\\
\texttt{ecology2} & $1.393\times 10^{-4}$  &  $1.247\times 10^{-4}$&$2.121 \times 10^{-4} $ &$ 6.428 \times 10^{-4}$\\
\texttt{bone010}  & $6.530 \times 10^{-6}$  & $1.233\times 10^{-5}$ &$8.338\times 10^{-6}$  &$2.297\times 10^{-4}$ \\
\texttt{af\_shell8}  & $4.948 \times 10^{-5}$& $5.643 \times 10^{-5}$ &$1.346\times 10^{-5}$ &$3.015\times 10^{-5}$\\
\hline
\end{tabular}%
}
\end{table}

\section*{Data availability}
Some of  the data used in our experiments were obtained from  the University of Florida(UFL) Sparse Matrix Collection \cite{davis2011university}.

\section*{Acknowledgements}
This work was made possible by a grant from Carnegie Corporation of New York (provided through the AIMS Research and Innovation Centre). The statements made and views expressed are solely the responsibility of the authors.

\bibliographystyle{plain}
\bibliography{main}

\end{document}